\documentclass[11pt]{amsart}
\usepackage{amssymb}
\usepackage{amsmath}
\usepackage{graphicx}
\usepackage{dcolumn}
\usepackage{hyperref}
\usepackage[utf8]{inputenc}
\usepackage[english]{babel}
\usepackage{amsthm}
\usepackage{latexsym}
\usepackage{xcolor}
\usepackage{flexisym}
\usepackage{ulem}

\usepackage{latexsym,amssymb,amsmath,amsfonts,bm}

\usepackage{subfig}
\usepackage{float}

\newtheorem{theorem}{Theorem}[section]
\newtheorem{corollary}[theorem]{Corollary}
\newtheorem{lemma}[theorem]{Lemma}
\newtheorem{remark}[theorem]{Remark}
\newtheorem{proposition}[theorem]{Proposition}
\newtheorem{definition}[theorem]{Definition}
\newtheorem{example}{Example}[section]

\newtheorem{algorithm}{Algorithm}[section]
\newtheorem{remarks}[theorem]{Remarks}
\newtheorem{conditions}[theorem]{Conditions}

\newcommand{\tab}[0]{\hspace{12pt}}

\newlength{\cellsize}
\newcommand\tableau[1]{
\vcenter{
\let\\=\cr
\baselineskip=-16000pt
\lineskiplimit=16000pt
\lineskip=0pt
\halign{&\tableaucell{##}\cr#1\crcr}}}

\newcommand{\tableaucell}[1]{{%
\def \arg{#1}\def \void{}%
\ifx \void \arg
\vbox to \cellsize{\vfil \hrule width \cellsize height 0pt}%
\else
\unitlength=\cellsize
\begin{picture}(1,1)
\put(0,0){\makebox(1,1){$#1$}}
\put(0,0){\line(1,0){1}}
\put(0,1){\line(1,0){1}}
\put(0,0){\line(0,1){1}}
\put(1,0){\line(0,1){1}}
\end{picture}%
\fi}}



\usepackage[left=1in, right=1in, top=1in, bottom=1in]{geometry}

\begin{document}

\title{On Combinatorial Models for Affine Crystals}

\author[C.~Briggs]{Carly Briggs}
\address[Carly Briggs]{Bennington College, One College Drive, Bennington, VT 05201, U.S.A.}
\email{carlybriggs@bennington.edu}

\author[C.~Lenart]{Cristian Lenart}
\address[Cristian Lenart]{Department of Mathematics and Statistics, State University of New York at Albany, 
Albany, NY 12222, U.S.A.}
\email{clenart@albany.edu}

\author[A.~Schultze]{Adam Schultze}
\address[Adam Schultze]{Department of Mathematics, Statistics, and Computer Science, St. Olaf College, 1520 St. Olaf Avenue, 
Northfield, MN 55057, U.S.A.}
\email{schult24@stolaf.edu}

\keywords{Kirillov-Reshetikhin crystal, Kashiwara-Nakashima column, quantum alcove model.}
\subjclass[2010]{Primary 05E10; Secondary 20G42.}

\maketitle

\begin{abstract}
The tableau model for Kirillov-Reshetikhin (KR) crystals, which are finite dimensional crystals corresponding to certain affine Lie algebras, is commonly used for its ease of crystal operator calculations.  However, its simplicity makes quite complex the calculation of statistics such as: keys (used to express Demazure characters), the crystal energy function (an affine grading on tensor products of KR crystals), and the combinatorial $R$-matrix (an affine crystal isomorphism permuting factors in a tensor product of KR crystals).  It has been shown that these calculations are much simpler with the added structure in the quantum alcove model for KR crystals. In this paper, we give an explicit description of the crystal isomorphism between the mentioned realizations of KR crystals in all classical Lie types.
\end{abstract}

%
%

\section{Introduction}

Kashiwara's crystals encode the structure of certain bases, called crystal bases \cite{Kashiwara 1991}, for highest weight representations of quantum groups $U_q(\mathfrak{g})$ as $q$ goes to zero \cite{Hong and Kang}.  In this paper, we will focus on crystals corresponding to representations of both simple and affine Lie algebras \cite{carter}. As combinatorial objects, they can be visualized as directed graphs, whose edges are given by Kashiwara operators which are derived from the Chevalley generators of the corresponding quantum group \cite{Bump and Schilling}.  The resulting combinatorics provides much insight into the original representations.  For instance, for irreducible representations of simple Lie algebras, by developing a tensor product rule for the Kashiwara operators, one can visualize the decomposition of the tensor product of such representations into irreducible components \cite{Hong and Kang}.

Given a dominant weight $\lambda$ for a simple Lie algebra, there is a corresponding connected crystal graph, $B(\lambda)$, whose vertices can be realized as semistandard Young tableaux (certain fillings of Young diagrams \cite{Fulton}). 
We can also consider crystals for affine Lie algebras. Here, we will focus on Kirillov-Reshetikhin (KR) crystals, which are finite affine connected crystals, but can be decomposed into disjoint unions of classical crystals upon the removal of the affine edges \cite{Kirillov Reshetikhin 1990}.  These crystals are indexed by $r\times s$ rectangles and are denoted $B^{r,s}$.  There are type specific models for KR crystals, like that given by Kashiwara-Nakashima (KN) tableaux \cite{FOS}, as well as type independent models such as the quantum alcove model, which works uniformly in all untwisted affine types \cite{Lenart Lubovsky 2015b,LNSSS 2016,LNSSS III}.
 We provide an explicit crystal isomorphism between the two models of the above mentioned KR crystals: the tableau model and the quantum alcove model.

 Lenart and Postnikov defined the so-called alcove model for highest weight crystals associated to a semisimple Lie algebra $\mathfrak{g}$. In fact, the model was defined more generally, for symmetrizable Kac-Moody algebras $\mathfrak{g}$  \cite{Lenart Postnikov 2007, Lenart Postnikov 2008}.  The alcove model is a discrete couterpart of the Littelmann path model. 

Lenart and Lubovsky then generalized the alcove model to one for Kirillov-Reshetikhin (KR) crystals of affine Lie algebras; this is known as the quantum alcove model \cite{Lenart Lubovsky 2015b}, as it is based on the quantum Bruhat graph. This graph first appeared in connection with the quantum cohomology of flag varieties of the corresponding finite Weyl group \cite{Fulton Woodward 2004}.  The path enumeration is determined by the choice of a certain sequence of alcoves, called an alcove path, like in the classical alcove model.  If we restrict to paths in the usual Bruhat graph, we recover the classical alcove model. The mentioned paths in the quantum Bruhat graph first appeared in \cite{Lenart 2012}, where they index the terms in the specialization to $t=0$ of the Ram-Yip formula \cite{Ram Yip 2011} for Macdonald polynomials $P_{\lambda}(X;q,t)$. 
Further, \cite{Lenart Lubovsky 2015b} defined crystal operators for the quantum alcove model, both the classical ones $f_i$, $i>0$, and the affine operator $f_0$.  It was shown in \cite{LNSSS 2016} that the quantum alcove model uniformly describes tensor products of column shape KR crystals for all untwisted affine types. An explicit crystal isomorphism was given in \cite{Lenart Lubovsky 2015b} for types $A$ and $C$ between the objects of the quantum alcove model and tensor products of Kashiwara-Nakashima (KN) columns \cite{Kashiwara Nakashima 1994}, using the bijections constructed in \cite{Lenart 2012}.

 
In recent years, many applications have resulted from the quantum alcove model.  In \cite{LNSSS 2016} it was shown that the so-called height statistic in the Ram-Yip formula mentioned above expresses the energy function on a tensor product of KR crystals, which endows it with an affine grading.  The translation of this energy function to the tableau model in type $C$ was then given in \cite{Lenart Schilling 2011}. On the other hand, extending Sch\"{u}tzenberger's \textit{jeu de taquin} on Young tableaux to the quantum alcove model  (which is based on so-called \textit{quantum Yang-Baxter moves}) results in a realization of the \textit{combinatorial $R$-matrix} \cite{Lenart Lubovsky 2015a}.  Further, the quantum alcove model was used to determine \textit{keys}, also known as \textit{initial direction} in the (quantum) LS path model. These detect Demazure crystals inside highest weight ones (in the alcove model), and Demazure-type crystals inside tensor products of KR crystals (in the quantum alcove model), see~\cite{LNSSS III}.  The computation of these keys is easy in the (quantum) alcove model and the bijection given in this paper then transfers the key computation from the (quantum) alcove model to the tableau model, where algorithms are only known for types $A$ and $C$ \cite{Santos, Santos2, Sheats}. 
      
While the tableau model is simpler, it has less easily accessible information, so it is generally hard to use in specific computations (like those of the energy function, the combinatorial $R$-matrix and keys, as described above). As these computations are much simpler in the quantum alcove model, an alternative is to relate them to the tableau model, via an explicit affine crystal isomorphism. As was mentioned above, this has been done in types $A$ and $C$. Here we extend this work to types $B$ and $D$, where the corresponding bijection is much more involved, and has important additional features.  

\subsection*{Acknowledgements.} 
C.B. was partially supported by the NSF grant DMS-1101264.
C.L. was partially supported by the NSF grants DMS-1362627 and DMS-1855592.
A.S. was partially supported by the NSF grant DMS-1362627 and the Chateaubriand Fellowship from the Embassy of France in the United States.

\section{Background}
\subsection{Root Systems}\label{section root systems}
 \cite{Humphreys} Let $\mathfrak{g}$ be a complex semisimple Lie algebra and $\mathfrak{h}$ a Cartan subalgebra, whose rank is $n$. 
  Let $\Phi\subset \mathfrak{h}^*$ be the corresponding irreducible 
  \textit{root system},
   $\mathfrak{h}^*_{\mathbb{R}}$ 
   the real span of the roots, and $\Phi^{+}\subset\Phi$ the set of positive roots. 

Let $\rho:=\frac{1}{2}(\sum_{\alpha\in\Phi^+}\alpha)$.
  Let $\alpha_1,\hdots,\alpha_n\in\Phi^+$ be the corresponding \textit{simple roots}. We denote $\langle\cdot,\cdot\rangle$ the nondegenerate scalar product on $\mathfrak{h}^*_{\mathbb{R}}$ induced by the Killing form.  Given a root $\alpha$, we consider the corresponding \textit{coroot} $\alpha^{\vee}:=2\alpha / \langle\alpha,\alpha\rangle$ and reflection $s_{\alpha}$.

Let $W$ be the corresponding \textit{Weyl group}, whose Coxeter generators are denoted, as usual, by $s_i:=s_{\alpha_i}$.  The length function on $W$ is denoted by $l(\cdot)$.  the \textit{Bruhat order} on $W$ is defined by its covers $w\lessdot ws_{\alpha},$ for $\alpha\in\Phi^+$, if $l(ws_\alpha) = l(w)+1$.  The mentioned covers correspond to the labeled directed edges of the \textit{Bruhat graph} on $W:$ \[w\xrightarrow{\alpha} ws_\alpha \hspace{8pt}\text{for} \hspace{8pt}w\lessdot ws_{\alpha}.\]\label{bruhat order graph eq}

The \textit{weight lattice} $\Lambda$ is given by \[\Lambda = \{\lambda\in\mathfrak{h}^*_{\mathbb{R}}: \langle\lambda,\alpha\rangle\in\mathbb{Z}\hspace{8pt}\text{for}\hspace{6pt}\text{any}\hspace{6pt}\alpha\in\Phi^+\}.\]

The weight lattice $\Lambda$ is generated by the \textit{fundamental weights} $\omega_1,\hdots,\omega_n$, which form the dual basis to the basis of simple coroots, i.e., $\langle\omega_i,\alpha^{\vee}\rangle=\delta_{i,j}$.  The set $\Lambda^+$ of \textit{dominant weights} is given by $$\Lambda^+:=\{\lambda\in\Lambda:\langle\lambda,\alpha^{\vee}\rangle\geq 0 \hspace{8pt}\text{for}\hspace{4pt}\text{any}\hspace{8pt}\alpha\in\Phi^+\}.$$

Given $\alpha\in\Phi$ and $k\in\mathbb{Z}$, we denote by $s_{\alpha,k}$ the reflection in the affine hyperplane $$H_{\alpha,k}:=\{\lambda\in\mathfrak{h}^*_{\mathbb{R}} : \langle\lambda,\alpha^{\vee}\rangle = k\}.$$ \label{affine hyperplane eq}

These reflections generate the \textit{affine Weyl Group} $W_{\text{aff}}$ for the \textit{dual root system} $\Phi^{\vee}:=\{\alpha^{\vee}|\alpha\in\Phi\}$. The hyperplanes $H_{\alpha,k}$ divide the real vector space $\mathfrak{h}^*_{\mathbb{R}}$ into open regions, called \textit{alcoves}.  The \textit{fundamental alcove} $A_{\circ}$ is given by 

$$A_{\circ}:= \{\lambda\in\mathfrak{h}^*_{\mathbb{R}} | 0<\langle\lambda,\alpha^{\vee}\rangle<1 \hspace{6pt}\text{for}\hspace{4pt}\text{all}\hspace{6pt}\alpha\in\Phi^+\}$$\label{Fund Alcove eq}

Define $w\triangleleft ws_{\alpha}$, for $\alpha\in\Phi^+$, if $l(ws_{\alpha}) = l(w) - 2\langle\rho,\alpha^{\vee}\rangle + 1$. These are the upward facing arrows in Figure~\ref{qbruhat graph}. The \textit{quantum Bruhat graph} \cite{Fulton Woodward 2004} is defined by adding to the Bruhat graph the following edges (the downward facing arrows in Figure~\ref{qbruhat graph}) labeled by positive roots $\alpha$:
$$w\xrightarrow{\alpha}ws_{\alpha}\hspace{8pt}\text{for}\hspace{6pt} w\triangleleft ws_{\alpha}.$$\label{quantum bruhat graph eq}

\begin{figure}
\centering
\includegraphics[scale=.35]{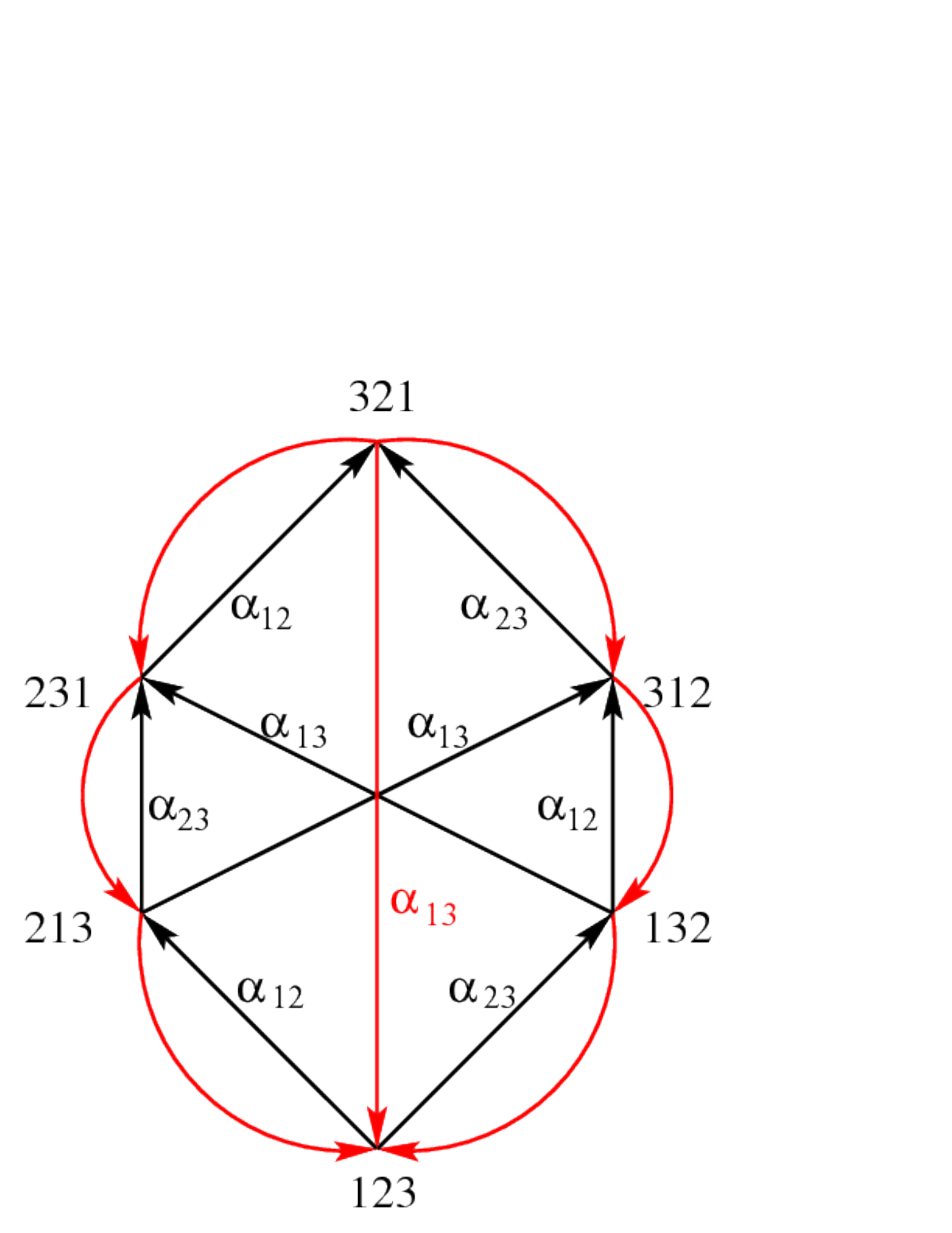}
\caption{The quantum Bruhat graph for the Weyl group of type $A_3$}\label{qbruhat graph}
\end{figure}

\subsection{Kirillov-Reshetikhin (KR) crystals}\label{KR section}
 
  Given a simple or affine Lie algebra $\mathfrak{g}$,  we define the corresponding \textit{Kashiwara Crystal} to be the basis resulting from taking the limit of the quantum group  $U_q(\mathfrak{g})$ as $q$ goes to zero.  Such bases can be given the structure of a colored oriented graph, which we call a \textit{Crystal Graph}, whose arrows are defined by the \textit{Kashiwara operators} which are related to the Chevalley generators \cite{Hong and Kang}.  We will now detail the resulting combinatorial structure of such crystal graphs.  Our definitions closely follow those given in \cite{Bump and Schilling}.
    
 \begin{definition}\label{combinatorial def of crystal}
 \end{definition}   Let $\Phi$ be the root system and $\Lambda$ the weight lattice, with simple roots $\alpha_i$ for $i$ in an indexing set $I$, associated to a Lie algebra $\mathfrak{g}$.  The \textit{crystal} of type $\Phi$ is then a nonempty set $B$ along with the following maps using $i\in I$ and an auxiliary element $0\notin B$:
    \begin{enumerate}
    \item the \textit{Kashiwara Operators} $e_i, f_i: B\rightarrow B \sqcup {0},$ with the conditions that if $x,y\in B$ then $e_i(x)=y$ if and only if $f_i(y) = x$.
    \item the $e_i$-$string$ and $f_i$-$string$ operators $\varepsilon_i, \varphi_i : B\rightarrow \mathbb{Z} \sqcup \{-\infty\},$ with the condition that if $x,y\in B$ are such that $e_i(x) = y$, then $\varepsilon_i(y) = \varepsilon(x) -1$ and $\varphi_i(y) = \varphi_i(x)+1$. In the case that $\varphi_i(x) = -\infty$ or $\varepsilon_i(x) = -\infty$, we require that $e_i(x) = f_i(x) = 0$.
    \item the \textit{weight} map $wt: B\rightarrow \Lambda,$ where if $x,y\in B$ are such that $e_i(x) = y$, then $wt(y) = wt(x) + \alpha_i$.  Further, for each $x\in B$ and $i\in I$, we have that $ \langle wt(x),\alpha_i^{\vee}\rangle = \varphi_i(x) - \varepsilon_i(x).$
    \end{enumerate}
    
The associated \textit{crystal graph} is then built with elements of $B$ as vertices and, for $x,y\in B$, we draw an edge $y\xrightarrow{i} x$ exactly when $f_i(y) = x$.

\vspace{12pt} Given two $\mathfrak{g}$-crystals $B_1$ and $B_2$, we define their tensor product $B_1\otimes B_2$ as follows.  As a set, $B_1\otimes B_2$ is the Cartesian product of the two sets. For $b=b_1\otimes b_2\in B_1\otimes B_2$, the weight function is simply $wt(b) = wt(b_1)+wt(b_2)$. The crystal operator $f_i$ is given by 
 
 $$f_i(b_1\otimes b_2) = \begin{cases}
     f_i(b_1)\otimes b_2 & \text{if}\hspace{5pt} \epsilon_i(b_1)\geq \varphi_i(b_2),  \\
     b_1\otimes f_i(b_2) & \text{otherwise,}
  \end{cases}$$
  
  while $e_i(b)$ is defined similarly.  
  
  \vspace{12pt} The \textit{highest weight crystal} $B(\lambda)$ of highest weight $\lambda\in\Lambda^+$ is a certain crystal with a unique element $\mu_{\lambda}$ such that $e_i(\mu_{\lambda}) = \textbf{0}$ for all $i\in I$ and $wt(\mu_{\lambda}) = \lambda$.  It encodes the structure of the crystal basis of the $U_q(\mathfrak{g})$-irreducible representation with highest weight $\lambda$ as $q$ goes to $0$ \cite{Bump and Schilling}.

 \vspace{12pt} A \textit{Kirillov-Reshetikhin (KR) crystal} \cite{Kirillov Reshetikhin 1990} is a finite crystal $B^{r,s}$ for an affine algebra, associated to a rectangle of height $r$ and length $s$.  We now describe the KR crystals $B^{r,1}$ for type $A_{n-1}^{(1)}$ (where $r\in\{1,2,\hdots,n-1\}$) as well as for types $B_n^{(1)}, C_n^{(1)},$ and $D_n^{(1)}$ (where $r\in \{1,2,\hdots,n\}$). As a classical type crystal (i.e. with the removal of the $\tilde{f_0}$ arrows) in types $A_{n-1}^{(1)}$ and $C_n^{(1)}$, we have that the KR crystal 
$B^{r,1}$ is isomorphic to the corresponding highest weight crystal $B(\omega_r)$. Similarly, in types $C_n^{(1)}$ and $D_n^{(1)}$, we have that the KR crystal $B^{r,1}$, as a classical type crystal, is isomorphic to the crystal $B(\omega_r) \sqcup B(\omega_{r-2}) \sqcup B(\omega_{r-4})\sqcup\hdots$
where each $B(\omega_k)$ is given by Kashiwara-Nakashima (KN) columns of height $k$ of the corresponding type.

\begin{definition}\label{def KN column}
  {\rm Kashiwara-Nakashima (KN) columns} of height $k$ are strictly increasing fillings of length $k$ columns with entries
 $\{1<2<\hdots <n\}$ in type $A_{n-1}$ and  with entries $\{1<\hdots <n<\overline{n} <\hdots <\overline{1}\}$
  in type $C_n$, $B_n$, and $D_n$ along  
  with the additional conditions:

\begin{enumerate}
\item The entries are strictly increasing from the top to bottom with the exception that:
\begin{enumerate}
\item the letter $0$ (ordered between $n$ and $\overline{n}$) can appear in type $B_n$ and can be repeated, and
\item the letters $n$ and $\overline{n}$ in type $D_n$ can alternate.
\end{enumerate}
\item If both letters $j$ and $\overline{\jmath}$ appear in the column, and $j$ is in the $a$-th box from the top and $\overline{\jmath}$ lies in the $b$-th box from the bottom, then $a+b\leq j$.
\end{enumerate}
\end{definition}

 \begin{example}
 A column of height 5 of type $B_6$

$$
\begin{array}{l}\tableau{{2}\\{3}\\{0}\\{0}\\{\overline{2}}} \end{array}
$$
 \end{example}
 
 For our purposes here, we would like to consider fillings of partition shapes, rather than just rectangular.  We define the associated crystal as follows.
  
 \begin{definition} For a partition $\textbf{p} = (p_1\geq p_2\geq\hdots\geq p_r)$, we define $$B^{\textbf{p}}:= B^{p_1,1}\otimes B^{p_2,1}\otimes\hdots\otimes B^{p_r,1}.$$
 \end{definition}
  
  \subsection{The quantum alcove model}
  We say that two alcoves are \textit{adjacent} if they are distinct and have a common wall.  Given a pair of adjacent alcoves $A$ and $B$, we write $A\xrightarrow{\beta} B$ if the common wall is of the form $H_{\beta,k}$ and the root $\beta\in\Phi$ points in the direction from $A$ to $B$.
  \begin{definition}{\rm \cite{Lenart Postnikov 2007}}\label{def alcove path}
  
  An {\rm alcove path} is a sequence of alcoves $(A_0,A_1,\hdots,A_m)$ such that $A_{j-1}$ and $A_j$ are adjacent, for $j=1,\hdots,m$.  We say that an alcove path is {\rm reduced} if it has minimal length among all alcove paths from $A_0$ to $A_m$.  
  \end{definition}
  
   \begin{definition}{\rm \cite{Lenart Postnikov 2007}}\label{def alcove lambda-chain}
  
  The sequence of roots $(\beta_1,\beta_2,\hdots,\beta_m)$ is called a {\rm $\lambda$-chain} if $$A_\circ = A_0\xrightarrow{-\beta_1} A_1\xrightarrow{-\beta_2} \hdots \xrightarrow{-\beta_m} A_m = A_{-\lambda}$$ is a reduced alcove path.
  \end{definition}
 
 We now fix the dominant weight $\lambda=\omega_{p_1}+\hdots + \omega_{p_r}$ and an alcove path $\Pi = (A_0,\hdots,A_m)$ from $A_0 = A_\circ$ to $A_m = A_{-\lambda}$. Note that $\Pi$ is determined by the corresponding $\lambda$-chain of positive roots $\Gamma:=(\beta_1,\hdots,\beta_m)$.  We let $r_i:=s_{\beta_i}$ and $J = \{j_1,j_2,\hdots,j_s\}\subseteq [m]$.  The elements of $J$ are called \textit{folding positions}.  We fold $\Pi$ in the hyperplanes corresponding to these positions and obtain a folded path.  Like $\Pi$, the folded path can be recorded by a sequence of roots, namely $\Delta = \Gamma(J) = (\gamma_1,\gamma_2,\hdots,\gamma_m)$; here $\gamma_k = r_{j_1}r_{j_2}\hdots,r_{j_p}(\beta_k)$, with $j_p$ the largest folding position less than $k$.  Given $i\in J$, we say that $i$ is a \textit{positive folding position} if $\gamma_i >0$, and a \textit{negative folding position} if $\gamma_i < 0 $.  We denote the positive folding positions by $J^+$ and the negative ones by $J^-$.  
 
 \begin{definition}{\rm \cite{Lenart Lubovsky 2015b}}\label{admissible subset}
 
 A subset $J = \{j_1<j_2<\hdots <j_s\}\subseteq [m]$ (possibly empty) is an {\rm admissible subset} if we have the following path in the quantum Bruhat graph on $W$:
 $$1\xrightarrow{\beta_{j_1}}r_{j_1}\xrightarrow{\beta_{j_2}}r_{j_1}r_{j_2}\xrightarrow{\beta_{j_3}}\hdots\xrightarrow{\beta_{j_s}}r_{j_1}r_{j_2}\hdots r_{j_s}.$$ We call $\Delta = \Gamma(J)$ an \textit{admissible folding}. Let $\mathcal{A} = \mathcal{A}(\mu)$ be the collection of all admissible subsets.
 \end{definition}
 
 See Example~\ref{Type A example} for an example of a $\Gamma$-chain and an admissible subset. 

\begin{theorem}{\rm \cite{LNSSS 2016}} The collection of all admissible subsets $\mathcal{A}(\lambda)$ is a combinatorial model for $B^{\textbf{p}}$.
\end{theorem}

\section{The bijection in types $A_{n-1}$ and $C_n$}
\subsection{The quantum alcove model and filling map in type $A_{n-1}$}
We start with the basic facts about the root system for type $A_{n-1}$.  We can identify the space $\mathfrak{h}^*_{\mathbb{R}}$ with the quotient $V:=\mathbb{R}^n/\mathbb{R}(1,\ldots,1)$, where $\mathbb{R}(1,\ldots,1)$ denotes the subspace in $\mathbb{R}^n$ spanned by the vector $(1,\ldots,1)$.  Let $\varepsilon_1,\ldots,\varepsilon_n\in V$ be the images of the coordinate vectors in $\mathbb{R}^n$.  The root system is $\Phi = \{\alpha_{ij} := \varepsilon_i-\varepsilon_j : i\neq j, 1\leq i,j \leq n\}$. The simple roots are $\alpha_i = \alpha_{i,i+1}$, for $i = 1,\ldots,n-1$.  The weight lattice is $\Lambda = \mathbb{Z}^n/\mathbb{Z}(1,\ldots,1)$.  The fundamental weights are $\omega_i = \varepsilon_1 + \varepsilon_2 + \ldots + \varepsilon_i$, for $i = 1,2,\ldots,n-1$. A dominant weight $\lambda = \lambda_1\varepsilon_1 + \ldots + \lambda_{n-1}\varepsilon_{n-1}$ is identified with the partition $(\lambda_1\geq\lambda_2\geq\ldots\geq\lambda_{n-1}\geq\lambda_n=0)$ having at most $n-1$ parts.  Note that $\rho = (n-1,n-2,\ldots,0)$.  Considering the Young diagram of the dominant weight $\lambda$ as a concatenation of columns, whose heights are $\lambda'_1,\lambda'_2,\ldots,$ corresponds to expressing $\lambda$ as $\omega_{\lambda'_1}+\omega_{\lambda'_2}+\ldots$ (as usual, $\lambda'$ is the conjugate partition to $\lambda$).

The Weyl group $W$ is the symmetric group $S_n$, which acts on $V$ by permuting the coordinates $\varepsilon_1,\ldots\,\varepsilon_n$.  Permutations $w\in S_n$ are written in one-line notation $w = w(1)\ldots w(n)$.  For simplicity, we use the same notation $(i,j)$, with $1\leq i < j \leq n$, for the positive root $\alpha_{ij}$ and the reflection $s_{\alpha_{ij}}$, which is the transposition $t_{ij}$ of $i$ and $j$.

We now consider the specialization of the quantum alcove model to type $A_{n-1}$. For any $k = 1,\ldots, n-1$, we have the following $\omega_k$-chain, denoted by $\Gamma(k)$ {\rm \cite{Lenart Postnikov 2008}}:
\begin{equation*}
\begin{array}{lllll}
(&\!\!\!\!(k,k+1),&(k,k+2),&\ldots,&(k,n)\,,\\
&&&\ldots\\
&\!\!\!\!(2,k+1),&(2,k+2),&\ldots,&(2,n)\,,\\
&\!\!\!\!(1,k+1),&(1,k+2),&\ldots,&(1,n)\,\,)\,.
\end{array}
\end{equation*}

We construct a $\lambda$-chain $\Gamma = (\beta_1,\beta_2,\ldots,\beta_m)$ as the concatenation  $\Gamma := \Gamma_1\ldots \Gamma_{\lambda_1}$, where $\Gamma_i := \Gamma(\lambda'_i)$.  Let $J = \{j_1<\ldots <j_s\}$ be a set of folding positions in $\Gamma$, not necessarily admissible, and let $T$ be the corresponding list of roots of $\Gamma$.  The factorization of $\Gamma$ induces a factorization on $T$ as $T = T_1T_2 \ldots T_{\lambda_1}$.  We denote by $T_1\ldots T_i$ the permutation obtained by multiplying the transpositions in $T_1,\ldots,T_i$ considered from left to right.  For $w\in W$, written $w = w_1 w_2 \ldots w_n$, 
let $w[i,j] = w_i\ldots w_j$.  To each $J$ we can associate a filling of a Young diagram $\lambda$, as follows.

\begin{definition}\label{Filling Map}
 Let $\pi_i = \pi_i(T) := T_1\ldots T_i$.  We define the {\rm filling map}, which produces a filling of the Young diagram $\lambda$, by $fill\_A(J) = fill\_A(T) := C_1\ldots C_{\lambda_1}$, where $C_i := \pi_i[1,\lambda'_i].$
We define the {\rm sorted filling map} $sfill\_A(J)$ to be the composition $sort\circ \mbox{fill\_A}(J)$, where {\rm sort} reorders increasingly each column of $fill\_A(J)$.
\end{definition}

\begin{definition}\label{circle order}
Define a {\rm circular order} $\prec_i$ on $[n]$ starting at $i$, by 
$$i\prec_i i+1\prec_i\ldots\prec_i n\prec_i 1\prec_i\ldots\prec_i i-1.$$  
\end{definition}

It is convenient to think of this order in terms of the numbers $1,\ldots,n$ arranged on a circle clockwise.  We make that convention that, whenever we write $a\prec b\prec c\prec\ldots$, we refer to the circular order $\prec = \prec_a$. Below is a criterion for the quantum Bruhat graph on the Weyl group $W$, ${\rm{QBG}}(W)$, in type $A_{n-1}$ using these orders. 

\begin{proposition}{\rm \cite{Lenart 2012}}\label{Type A QB criterion}
 For $1\leq i<j\leq n$, we have an edge $w\xrightarrow{(i,j)} w(i,j)$ in ${\rm{QBG}}(W)$ if and only if there is no $k$ such that $i<k<j$ and $w(i)\prec w(k)\prec w(j)$.
\end{proposition}

\begin{example}\label{Type A example}
{\rm 
Consider the dominant weight $\lambda = 3\varepsilon_1 +2\varepsilon_2 = \omega_1+2\omega_2$  in the root system $A_2$, which corresponds to the Young diagram $\begin{array}{l} \tableau{{}&{}&{}\\ {}&{}} \end{array}$.  The corresponding $\lambda$-chain is $$\Gamma = \Gamma_1\Gamma_2\Gamma_3 = \Gamma(2)\Gamma(2)\Gamma(1)= \{\underline{(2,3)},\underline{(1,3)}|\underline{(2,3)},(1,3)|\underline{(1,2)},(1,3)\}\,.$$ 
Consider $J=\{1,2,3,5\}$, cf. the underlined roots, with 
$T = \{(2,3),(1,3)|(2,3)|(1,2)\}.$

We write the permutations in Definition~\ref{admissible subset} as broken columns.  Note that $J$ is admissible since, based on Proposition~\ref{Type A QB criterion}, we have
\begin{equation*}
\begin{array}{l} \tableau{{1}\\ {\textbf{2}}}\\ \\
\tableau{{\textbf{3}}} \end{array}
\begin{array}{l} \lessdot \end{array}
\begin{array}{l} \tableau{{\textbf{1}}\\ {3}}\\ \\
\tableau{{\textbf{2}}} \end{array}
\begin{array}{l} \lessdot \end{array}
\begin{array}{l} \tableau{{2}\\ {3}}\\ \\
\tableau{{1}} \end{array}
\:|\:
\begin{array}{l} \tableau{{{2}}\\ {\textbf{3}}}\\ \\
\tableau{{\textbf{1}}} \end{array}
\begin{array}{l} \triangleleft \end{array}
\begin{array}{l} \tableau{{2}\\ {1}}\\ \\
\tableau{{3}} \end{array}
\:|\:
\begin{array}{l} \tableau{{\textbf{2}}}\\ \\
\tableau{{\textbf{1}}\\ {3}} \end{array}
\begin{array}{l} \triangleleft \end{array}
\begin{array}{l} \tableau{{1}}\\ \\
\tableau{{2}\\ {3}} \end{array}
\:|\:
,
\end{equation*} 

where the symbols $\lessdot$ and $\triangleleft$ signify Bruhat coverings as given in Section~\ref{section root systems}.  By considering the top part of the last column in each segment, and by concatenating these columns left to right, we obtain $ fill\_A(J) = \begin{array}{l} \tableau{{2}&{2}&{1}\\ {3}&{1}} \end{array}$ and $ sfill\_A(J)= \begin{array}{l} \tableau{{2}&{1}&{1}\\ {3}&{2}} \end{array}$.
}\end{example}

\begin{theorem}{\rm \cite{Lenart 2012,Lenart Lubovsky 2015b}} The map ``$sfill\_A$'' is an affine crystal isomorphism between $\mathcal{A} (\lambda)$ and 
$B^{\lambda'}:=B^{\lambda_1',1}\otimes B^{\lambda_2',1}\otimes\ldots$.
\end{theorem}

The proof of bijectivity is given in~\cite{Lenart 2012} by constructing an inverse map.  We will now present the algorithm for constructing this map, as the corresponding construction in the other classical types is based on this algorithm.

\subsection{The inverse map in type $A_{n-1}$}

Consider $B^{{\lambda'}}:= B^{\lambda_1',1}\otimes B^{\lambda_2',1}\otimes\ldots =  B(\omega_{\lambda_1'})\otimes B(\omega_{\lambda_2'})\otimes\ldots$.  This is simply the set of column-strict fillings of the Young diagram $\lambda$ with integers in $[n]$. Fix a filling $b$ in $B^{{\lambda'}}$ written as a concatenation of columns $b_1\ldots b_{\lambda_1}$.

The algorithm for mapping $b$ to a sequence of roots $S\subset \Gamma$ consists of two sub-algorithms, which we call the \textit{Reorder algorithm} (this reverses the ordering of columns $b_i$ from \textit{sort} back to that of the corresponding column in the $fill\_A$ map) and the \textit{Path algorithm} (this provides the corresponding path in the quantum Bruhat graph). 

The Reorder algorithm (Algorithm~\ref{Reorder algorithm}) takes $b$ as input and outputs a filling $ord\_A(b) = C$, a reordering of the column entries, based on the circle order given in Definition~\ref{circle order}.

\begin{algorithm}\label{Reorder algorithm}
(``ord\_A'')

 let $C_1:=b_1$;

 \hspace{8pt}for $i$ from $2$ to $\lambda_1$ do

 \hspace{16pt} for $j$ from $1$ to $\lambda'_i$ do

 \hspace{24pt} let $C_i(j):=min_{\prec_{C_{i-1}(j)}}(b_i\setminus \{C_i(1),\ldots,C_i(j-1)\})$

 \hspace{16pt} end do;

 \hspace{8pt} end do;

return $C:=C_1\ldots C_{\lambda_1}.$
\end{algorithm}

\begin{example}{\rm 
Algorithm~\ref{Reorder algorithm} gives the filling $C$ from $b$ below.

 $$b  = \tableau{{3}\\{5}\\{6}}  \tableau{{2}\\{3}\\{4}}\tableau{{1}\\{2}\\{4}}\tableau{{2}\\ \\ \\} \xrightarrow{ord\_A} \tableau{{3}\\{5}\\{6}}  \tableau{{3}\\{2}\\{4}}\tableau{{4}\\{2}\\{1}}\tableau{{2}\\ \\ \\} = C $$
}\end{example}

The path algorithm (Algorithm~\ref{Greedy algorithm}) takes the reordered filling $C$ and outputs a sequence of roots $Path\_A(C) = S\subset \Gamma$. Let $C_0$ be the increasing column filled with $1,2,\ldots,n$.

\begin{algorithm}\label{Greedy algorithm}
(``Path\_A'')
  
 for $i$ from $1$ to $\lambda_1$ do
 
 \hspace{12pt} let $S_i:=\emptyset$, $A := C_{i-1}$; 

 \hspace{12pt} for $(l,m)$ in $\Gamma_i$ do

 \hspace{24pt} if $A(l)\neq C_i(l)$ and $A(l)\prec A(m)\prec C_i(l)$ then let $S_i:=S_i,(l,m)$ and $A:=A(l,m)$;
 
 \hspace{24pt} end if;

 \hspace{12pt} end do;
 
 end do;

 return $S := S_1\ldots S_{\lambda_1}$.
\end{algorithm}

\begin{example}{\rm 
Consider $b=\tableau{{1}&{1}&{2}\\{3}&{2}&\\{4}&&}\in B^{(3,2,1)}$, where $\lambda=\lambda' = (3,2,1)$ and $n=4$.  We have  $$\Gamma = \Gamma(3)\Gamma(2)\Gamma(1) = \{(3,4),(2,4),(1,4)|(2,3),(2,4),(1,3),(1,4)|(1,2),(1,3),(1,4)\}.$$

Notice that $ord\_A(b) = b$, and that $Path\_A\circ ord\_A(b)$ outputs

 $S=S_1S_2S_3 = \{(3,4),(2,4)|(2,3),(2,4)|(1,2)\}$ via the following path in $\mbox{QBG}(W)$:
$$\begin{array}{l}\tableau{{1}\\{2}\\{ 3}} \\ \\ \tableau{{ 4}} \end{array} \!
\begin{array}{c} \\ \xrightarrow{(3,4)} \end{array}\! 
\begin{array}{l}\tableau{{ 1}\\{ 2}\\{ 4} \\ \\ {  3}} \end{array} \begin{array}{c} \\ {\xrightarrow{(2,4)}} 
\end{array}\! \begin{array}{l}\tableau{ {1}\\{ 3}}  \\ \tableau{{ 4}\\ \\{ 2}} 
\end{array}\!  \:|\:  \begin{array}{l}\tableau{{{ 1}}\\{{ 3}}} \\ \\ 
\tableau{{{4}}\\{ 2}}\end{array}\!\begin{array}{c} \\ \xrightarrow{(2,3)} 
\end{array}\!  \begin{array}{l}\tableau{{ 1}\\{ 4}}\\ \\ \tableau{{3}\\{ 2}}
\end{array} \begin{array}{c} \\ \xrightarrow{(2,4)} \end{array}\!  \begin{array}{l}
\tableau{{ 1}\\{ 2}}\\ \\ \tableau{{3}\\{4}}\end{array}  \:|\: 
\begin{array}{l}\tableau{{ 1}\\ \\ {{ 2}}\\{3}\\{4}} \end{array} 
\!\begin{array}{c} \\ \xrightarrow{(1,2)} \end{array}\!   \begin{array}{l} \tableau{{ 2}\\ \\{1}\\{3}} \\ \tableau{{4}}\end{array} \!  \, .$$
}\end{example}

\begin{theorem}{\rm \cite{Lenart 2012}}
If $\mbox{fill\_A}(T)=C$, then the output of the Greedy algorithm $C\mapsto S$ is such that $S = T$. Moreover,  the map $``Path\_A\circ ord\_A''$ is the inverse of ``$\mbox{sfill\_A}$''. 
\end{theorem}

\subsection{The quantum alcove model and filling map in type $C_n$}\label{type C setup}
We start with the basic facts about the root system for type $C_{n}$.  We can identify the space $\mathfrak{h}^*_{\mathbb{R}}$ with $V:=\mathbb{R}^n$, with coordinate vectors $\varepsilon_1,\ldots,\varepsilon_n\in V$.  The root system is $\Phi = \{\pm\varepsilon_i\pm\varepsilon_j \,:\, i\neq j,\, 1\leq i<j \leq n\}\cup\{\pm 2\varepsilon_i \,:\, 1\leq i \leq n\}$. 

The Weyl group $W$ is the group of signed permutations $B_n$, which acts on $V$ by permuting the coordinates and changing their signs.  A signed permutation is a bijection $w$ from $[\overline{n}]:=\{1<2<\ldots <n<\overline{n}<\overline{n-1}<\ldots <\overline{1\}}$ to $[\overline{n}]$ which satisfies $w(\overline{\imath}) = \overline{w(i)}$.  Here, $\overline{\imath}$ is viewed as $-i$, so that $\overline{\overline{\imath}} = i$, and we can define $|i|$ and $sign(i)\in\{\pm 1\}$, for $i\in[\overline{n}]$.  We will use the so-called \textit{window notation} $w = w(1)w(2)\ldots w(n)$.  For simplicity, given $1\leq i<j\leq n$, we denote by $(i,j)$ and $(i,\overline{\jmath})$ the roots $\varepsilon_i-\varepsilon_j$ and $\varepsilon_i+\varepsilon_j$, respectively; the corresponding reflections, denoted in the same way, are identified with the composition of transpositions  $t_{ij}t_{\overline{\jmath}\overline{\imath}}$ and $t_{i\overline{\jmath}}t_{j\overline{\imath}}$, respectively.  Finally, we denote by $(i,\overline{\imath})$ the root $2\varepsilon_i$ and the corresponding reflection, identified with the transposition $t_{i\overline{\imath}}$.

We now consider the specialization of the quantum alcove model to type $C_n$.  For any $k = 1,\ldots,n$, we have the following (split) $\omega_k$-chain, denoted by $\Gamma^l(k)\Gamma^r(k)$ \cite{Lenart 2012}, where:
\begin{equation}\label{omegakchain}\Gamma^l(k):= \Gamma^{kk}\ldots \Gamma^{k1}, \hspace{8pt} \Gamma^r(k):=\Gamma^k\ldots \Gamma^2\,,\end{equation}
\vspace{-12pt}
\begin{equation*}
\begin{array}{lllll}
\;\;\;\;\;\;\;\;\;\;\,\Gamma^{ki}:=(
&\!\!\!\! (i,k+1),&(i,k+2),&\ldots,&(i,n)\,,\\
&\!\!\!\! (i,\overline{\imath})\,,\\
&\!\!\!\! (i,\overline{n}),&(i,\overline{n-1}),&\ldots,&(i,\overline{k+1})\,,\\
&\!\!\!\! (i,\overline{i-1}),&(i,\overline{i-2}),&\ldots,&(i,\overline{1})\:)\,,
\end{array}
\end{equation*}
\vspace{-9pt}
$$\!\!\!\!\!\Gamma^{i}:=((i,\overline{i-1}),(i,\overline{i-2}),\ldots,(i,\overline{1}))\,.$$
We refer to the four rows above in $\Gamma^{ki}$ as stages I, II, III, and IV respectively.
We can construct a $\lambda$-chain as a concatenation $\Gamma:=\Gamma_{1}^l\Gamma_{1}^r\ldots \Gamma_{\lambda_1}^l\Gamma_{\lambda_1}^r$, where  $\Gamma^l_i:=\Gamma^l(\lambda'_i)$ and $\Gamma^r_i:=\Gamma^r(\lambda'_i)$.  We will use interchangeably the set of positions $J$ in the $\lambda$-chain $\Gamma$ and the sequence of roots $T$ in $\Gamma$ in those positions, which we call a \textit{folding sequence}.  The factorization of $\Gamma$ with factors  $\Gamma^l_i$,$\Gamma^r_i$ induces a factorization of $T$ with factors $T^l_i$,$T^r_i$. We define the circle order $\prec_a$ in a similar way to Definition~\ref{circle order}, but on the set $[\overline{n}]$. Below is a criterion for ${\rm{QBG}}(W)$ in type $C_n$, analogous to Proposition~\ref{Type A QB criterion}.

\begin{proposition}\label{type C bruhat conditions}{\rm \cite{Lenart 2012}} 
The quantum Bruhat graph of type $C_n$ has the following edges:
\begin{enumerate}
\item given $1\leq i < j\leq n$, we have an edge  $w\xrightarrow{(i,j)} w(i,j)$ if and only if there is no $k$ such that $i<k<j$ and $w(i)\prec w(k)\prec w(j)$;
\item given $1\leq i < j\leq n$, we have an edge $w\xrightarrow{(i,\overline{\jmath})} w(i,\overline{\jmath})$ if and only if $w(i)<w(\overline{\jmath})$, $sign(w(i))=sign(w(\overline{\jmath})$, and there is no $k$ such that $i<k<\overline{\jmath}$ and $w(i)\prec w(k)\prec w(\overline{\jmath})$;
\item given $1\leq i \leq n$, we have an edge $w\xrightarrow{(i,\overline{\imath})} w(i,\overline{\imath})$ if and only if there is no $k$ such that $i<k<\overline{\imath}$ (or equivalently, $i<k\leq n$) and $w(i)\prec w(k)\prec w(\overline{\imath})$.
\end{enumerate}
\end{proposition}

\begin{definition}\label{deffillc}
Given a folding sequence $T$, we consider the signed permutations
$\pi^l_i:=T_{1}^lT_1^r\ldots T_{i-1}^lT_{i-1}^rT^l_i$, $\pi^r_i:=\pi^l_iT^r_i.$
Then the {\rm filling map} is the map ``$\mbox{fill\_C}$'' from folding sequences $T$ in $\mathcal{A}(\lambda)$ to fillings $\mbox{fill\_C}(T) = C^l_{1}C^r_{1}\ldots C^l_{\lambda_1}C^r_{\lambda_1}$ of the shape $2\lambda$, which are viewed as concatenations of columns; here  $C^l_i:=\pi^l_i[1,\lambda'_i]$ and $C^r_i:=\pi^r_i[1,\lambda'_i]$, for $i=1,\ldots,\lambda_1.$ 
We then define $\mbox{sfill\_C}: \mathcal{A}(\lambda)\rightarrow B^{\lambda'}$ to be the composition ``$\mbox{sort}\circ \mbox{fill\_C}$'', where ``{sort}'' reorders the entries of each column increasingly; here we represent a KR crystal $B^{r,1}$ as a {\rm split} (also known as {\rm doubled}) KN column of height $r$, see Section~{\rm \ref{invc}}.
\end{definition}

\begin{theorem}{\rm \cite{Lenart 2012,Lenart Lubovsky 2015b}}
The map ``$sfill\_C$'' is an affine crystal isomorphism between $\mathcal{A} (\lambda)$ and 
$B^{\lambda'}$.
\end{theorem}

\subsection{The inverse map in type $C_n$}\label{inverse map type C section}\label{invc}

Recall from the construction of the filling map in type $A_{n-1}$ that we treated the columns of a filling as initial segments of permutations. However, the KN columns of type $C_n$ allow for both $i$ and $\overline{\imath}$ to appear as entries in such a column.  In order to pursue the analogy with type $A_{n-1}$, cf. Definition~\ref{deffillc}, we need to replace a KN column with its \textit{split} version, i.e., two columns of the same height as the initial column. The splitting procedure, described below, gives an equivalent definition of KN columns, see Section~\ref{KR section}.

\begin{definition}\label{type C splitting}{\rm \cite{Lecouvey 2002}}
Let $C$ be a column and $I=\{z_1>\ldots >z_r\}$ be the set of unbarred letters $z$ such that the pair $(z,\overline{z})$ occurs in $C$. The column $C$ can be split when there exists a set of $r$ unbarred letters $J=\{t_1>\ldots>t_r\}\subset [n]$ such that 
$t_1$ is the greatest letter in $[n]$ satisfying: $t_1<z_1, t_1\notin C$, and $\overline{t_1}\notin C$, and for $i=2,\ldots,r$, the letter $t_i$ is the greatest value in $[n]$ satisfying $t_i<min(t_{i-1},z_i),t_i\notin C$, and $\overline{t_i}\notin C$.
In this case we write:
\begin{enumerate}
\item $rC$ for the column obtained by changing $\overline{z_i}$ into $\overline{t_i}$ in $C$ for each letter $z_i\in I$, and by reordering if necessary,
\item $lC$ for the column obtained by changing $z_i$ into $t_i$ in $C$ for each letter $z_i\in I$, and by reordering if necessary.
\end{enumerate}
The pair $(lC,rC)$ is then called a {\rm split} (or {\rm doubled}) column.
\end{definition}

Given our fixed dominant weight $\lambda$, an element $b$ of $B^{\lambda'}$ can be viewed as a concatenation of KN columns $b_1\ldots b_{\lambda_1}$, with $b_i$ of height $\lambda_i'$. Let $b':=b^l_1b^r_1\ldots b^l_{\lambda_1}b^r_{\lambda_1}$ be the associated filling of the shape $2\lambda$, where  $(b^l_i,b^r_i) := (lb_i,rb_i)$ is the  splitting of the KN column $b_i$.

\vspace{12pt}The algorithm for mapping $b'$ to a sequence of roots $S\subset \Gamma$ is similar to the type $A_{n-1}$ one.  The Reorder algorithm $``ord\_C''$ for type $C_n$ is the obvious extension from type $A_{n-1}$.  The path algorithm $``Path\_C''$ is also similar to its type $A_{n-1}$ counterpart, but merits discussion. Recall that an $\omega_k$-chain in type $C_n$ factors as $\Gamma^l(k)\Gamma^r(k)$.  While the path algorithm parses through $\Gamma^l(k)$, it outputs a chain from the previous right column to the current left column reordered.  While the path algorithm parses through $\Gamma^r(k)$, it outputs a chain between the current left and right columns, both reordered.

\begin{theorem}{\rm \cite{Lenart 2012}} 
The map ``$Path\_C\circ ord\_C\circ split\_C$'' is the inverse of the type $C_n$ ``$sfill\_C$'' map. 
\end{theorem}

\section{The bijection in type $B_n$}\label{B}
We now move to the main content of this paper: extending the work done in types $A_{n-1}$ and $C_n$ to both types $B_n$ and $D_n$.  The filling map naturally extends to all classical types, however the corresponding inverse maps become more interesting as we progress to type $B_n$, and further still with type $D_n$.  The changes in the inverse maps are direct consequences of differences between the corresponding structure of the KN columns, as well as differences in the quantum Bruhat graphs. 

\subsection{The type $B_n$ Kirillov-Reshetikhin crystals}\label{type b kr section}

We begin by recalling the basic facts of the type $B_n$ root system.  Similar to type $C_n$, we can identify the space $\mathfrak{h}^*_{\mathbb{R}}$ with $V:=\mathbb{R}^n$, with coordinate vectors $\varepsilon_1,\ldots,\varepsilon_n\in V$.  The root system is $\Phi = \{\pm\varepsilon_i\pm\varepsilon_j \,:\, i\neq j,\, 1\leq i<j \leq n\}\cup\{\pm \varepsilon_i \,:\, 1\leq i \leq n\}$. 

The Weyl group $W$ is the group of signed permutations $B_n$, which acts on $V$ by permuting the coordinates and changing their signs.  We again note that a signed permutation is a bijection $w$ from $[\overline{n}]:=\{1<2<\ldots <n<\overline{n}<\overline{n-1}<\ldots <\overline{1\}}$ to $[\overline{n}]$ which satisfies $w(\overline{\imath}) = \overline{w(i)}$.  Here, $\overline{\imath}$ is viewed as $-i$, so that $\overline{\overline{\imath}} = i$, and we can define $|i|$ and $sign(i)\in\{\pm 1\}$, for $i\in[\overline{n}]$ in the obvious way.  We will use the so-called \textit{window notation} $w = w(1)w(2)\ldots w(n)$.  For simplicity, given $1\leq i<j\leq n$, we denote by $(i,j)$ and $(i,\overline{\jmath})$ the roots $\varepsilon_i-\varepsilon_j$ and $\varepsilon_i+\varepsilon_j$, respectively; the corresponding reflections, denoted in the same way, are identified with the composition of transpositions  $t_{ij}t_{\overline{\jmath}\overline{\imath}}$ and $t_{i\overline{\jmath}}t_{j\overline{\imath}}$, respectively.  Finally, we denote by $(i,\overline{\imath})$ the root $\varepsilon_i$ and the corresponding reflection, identified with the transposition $t_{i\overline{\imath}}$.

\vspace{12pt}
Recall from Section~\ref{KR section} that, given a fixed dominant weight $\lambda$, we can write
 \[B^{\lambda'} = \bigotimes\limits^1_{i = \lambda_1}B^{\lambda'_i,1},\]
   where each $B^{r,1}$ is a column shape type $B_n$ Kirillov-Reshitihkin crystal. When viewed as a classical type crystal, we have \[B^{r,1}\cong B(\omega_r) \sqcup B(\omega_{r-2}) \sqcup B(\omega_{r-4})\sqcup\ldots\] where, as before, the elements of the set $B(\omega_k)$ are given by KN columns of height $k$.  As in type $C_n$, the KN columns in type $B_n$ are allowed to contain both $i$ and $\overline{\imath}$ values; they may also contain the value $0$.  This is addressed in the type $B_n$ splitting algorithm ``{\it split\_B}''  by adding the $0$ values in the column to the set $I$ (see Definition~\ref{type C splitting}), and then by proceeding as in type $C_n$ {\rm \cite{lecsbd}}. As it was in Type $C_n$, it will be usefull to realize the tensor factors $B^{k,1}$ in terms of split columns of height $k$.

\subsection{The quantum alcove model and filling map in type $B_n$}\label{type B alcove def section}

 We now consider the specialization of the quantum alcove model to type $B_n$.  For any $k = 1,\ldots,n$, we define the following (split) $\omega_k$-chain, denoted by $\Gamma^l(k)\Gamma^r(k)$ \cite{Lenart 2012}, similarly to type $C_n$ where:
\begin{equation} \Gamma^l(k):= \Gamma^{kk}\ldots \Gamma^{k1}, \hspace{8pt} \Gamma^r(k):=\Gamma^k\ldots \Gamma^2\,,\end{equation}
\vspace{-12pt}
\begin{equation*}
\begin{array}{lllll}
\;\;\;\;\;\;\;\;\;\;\,\Gamma^{ki}:=(
&\!\!\!\! (i,k+1),&(i,k+2),&\ldots,&(i,n)\,,\\
&\!\!\!\! (i,\overline{\imath})\,,\\
&\!\!\!\! (i,\overline{n}),&(i,\overline{n-1}),&\ldots,&(i,\overline{k+1})\,,\\
&\!\!\!\! (i,\overline{i-1}),&(i,\overline{i-2}),&\ldots,&(i,\overline{1})\:)\,,
\end{array}
\end{equation*}
\vspace{-9pt}
$$\!\!\!\!\!\Gamma^{i}:=((i,\overline{i-1}),(i,\overline{i-2}),\ldots,(i,\overline{1}))\,.$$
We will continue to refer to the four rows above in $\Gamma^{ki}$ as stages I, II, III, and IV respectively (c.f. Figure~\ref{stages_fig}).
We can construct a $\lambda$-chain as a concatenation $\Gamma:=\Gamma_{1}^l\Gamma_{1}^r\ldots \Gamma_{\lambda_1}^l\Gamma_{\lambda_1}^r$, where  $\Gamma^l_i:=\Gamma^l(\lambda'_i)$ and $\Gamma^r_i:=\Gamma^r(\lambda'_i)$.  We will use interchangeably the set of positions $J$ in the $\lambda$-chain $\Gamma$ and the sequence of roots $T$ in $\Gamma$ in those positions, which we call a \textit{folding sequence}.  The factorization of $\Gamma$ with factors  $\Gamma^l_i$,$\Gamma^r_i$ induces a factorization of $T$ with factors $T^l_i$,$T^r_i$. We use the same circle order $\prec_a$ on the set $[\overline{n}]$ as the one in type $C_n$.

 The following are conditions on the quantum Bruhat graph of type $B_n$. 
\begin{proposition}\label{type B bruhat conditions}{\rm \cite{Briggs}} 
The quantum Bruhat graph of type $B_n$ has the following edges.
\begin{enumerate}
\item Given $1\leq i < j\leq n$, we have an edge $w\xrightarrow{(i,j)} w(i,j)$ if and only if there is no $k$ such that $i<k<j$ and $w(i)\prec w(k)\prec w(j)$.
\item Given $1\leq i<j\leq n$, we have an edge $w\xrightarrow{(i,\overline{\jmath})} w(i,\overline{\jmath})$ if and only if one of the following conditions holds:
\begin{enumerate}
\item $w(i)<w(\overline{\jmath})$, $sign(w(i))=sign(w(\overline{\jmath}))$, and there is no $k$ such that $i<k<\overline{\jmath}$ and $w(i)< w(k)< w(\overline{\jmath})$;
\item $sign(w(i))=-1$, $sign(w(\overline{\jmath}))=1$, and there is no $k$ such that $i<k\neq j< \overline{\jmath}$ and $w(i)\prec w(k)\prec w(\overline{\jmath})$.
\end{enumerate}
 \item Given $1\leq i\leq n$, we have an edge $w\xrightarrow{(i,\overline{\imath})} w(i,\overline{\imath})$ if and only if:
\begin{enumerate}
\item $w(i)<w(\overline{\imath})$ and there is no $k$ such that $i<k<\overline{\imath}$ and $w(i)\prec w(k)\prec w(\overline{\imath})$;
\item or $w(\overline{\imath})<w(i)$ and $i=n$.
\end{enumerate}
\end{enumerate}
\end{proposition}

\begin{figure}[H]
\centering
\includegraphics[scale=.55]{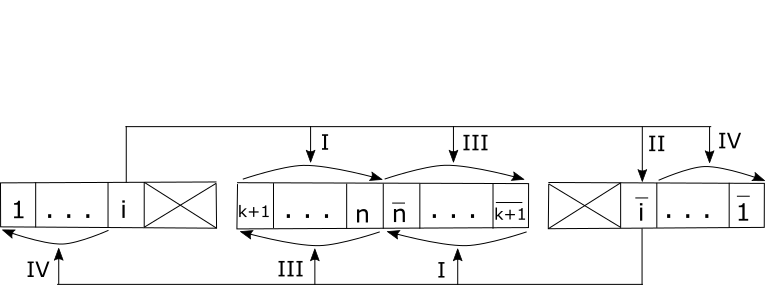}
\caption{A visualization of the four stages of roots in $\Gamma^{ki}$.}\label{stages_fig}
\end{figure}

Note that there are two major differences from the type $C_n$ quantum Bruhat graph criterion.  

\begin{enumerate}
\item Since the root $(n,\overline{n})$ is not in $\Gamma(k)$ for any $k<n$, we will never have the case $3b$.  This means that we lose the ability to apply the transposition $t_{i\overline{\imath}}$ when $w(\overline{\imath})<w(i)$.
\item In return, we gain some extra transpositions through $2b$, as there are now cases where the quantum Bruhat graph criterion allows for an arrow $\xrightarrow{(i,\overline{\jmath})}$ when $w(i)>w(\overline{\jmath})$.
\end{enumerate}

We now provide a description of the filling map in type $B_n$.  Note that it is the natural extension from type $C_n$.

\begin{definition}\label{deffillb}
Given a folding sequence $T$, we consider the signed permutations
$\pi^l_i:=T_{1}^lT_1^r\ldots T_{i-1}^lT_{i-1}^rT^l_i$, $\pi^r_i:=\pi^l_iT^r_i.$
Then the {\rm filling map} is the map ``$\mbox{fill\_B}$'' from folding sequences $T$ in $\mathcal{A}(\lambda)$ to fillings $\mbox{fill\_B}(T) = C^l_{1}C^r_{1}\ldots C^l_{\lambda_1}C^r_{\lambda_1}$ of the shape $2\lambda$, which are viewed as concatenations of columns; here  $C^l_i:=\pi^l_i[1,\lambda'_i]$ and $C^r_i:=\pi^r_i[1,\lambda'_i]$, for $i=1,\ldots,\lambda_1.$ 
We then define $\mbox{sfill\_B}: \mathcal{A}(\lambda)\rightarrow B^{\lambda'}$ to be the composition ``$\mbox{sort}\circ \mbox{fill\_B}$'', where ``{sort}'' reorders the entries of each column increasingly; here we represent a KR crystal $B^{k,1}$ as a {\rm split} (also known as {\rm doubled}) KN column of height $k$, see Section~{\rm \ref{type b kr section}}.
\end{definition}


\subsection{The type $B_n$ inverse map }

The main process remains similar to that of the type $C_n$ inverse map. However, the differences in the type $B_n$ quantum Bruhat graph, $B^{r,1}$ Kirillov-Reshetikhin crystals, and $KN$ columns add key new features.  The fact that type $B_n$ columns can contain the value $0$ was acknowledged in Section~\ref{type b kr section} with the revised algorithm $split\_B$.  We now address the difficulties associated to the QBG criterion and KR crystals.

\vspace{12pt} Recall that the KR crystals of column shape can be written as $B^{k,1}\cong B(\omega_k) \sqcup B(\omega_{k-2}) \sqcup B(\omega_{k-4})\sqcup\ldots$ where the elements of the set $B(\omega_r)$ are given by KN columns of height $r$.  This means that each $B^{k,1}$ contains columns of height less than $k$.  We need to extend them to full height $k$ so that the transpositions of the corresponding $\Gamma^l(k)\Gamma^r(k)$ may be correctly applied. The respective algorithm ``{\it extend}'' is given below.

\begin{algorithm}{\rm \cite{Briggs}}
Given a split column $(lC,rC)$ of length $1\leq r<n$ and $r\leq k<n$, append $\{\overline{\imath}_1<\ldots<\overline{\imath}_{r-k}\}$ to $lC$ and $\{i_1<\ldots<i_{r-k}\}$ to $rC$, where $i_1$ is the minimal value in $[\overline{n}]$ such that $i_1,\overline{\imath}_1\notin lC,rC$, and $i_t$ for $2\leq t\leq r-k$ is minimum value in $[\overline{n}]$ such that $i_t,\overline{\imath}_t\notin lC,rC$ and $i_t>i_{t-1}$.  Sort the extended columns increasingly. Let $(\widehat{lC},\widehat{rC})$ be the {\rm extended split column}.
\end{algorithm}

\begin{example}
The following is a $KN$ column $C$ of type $B_8$ with its split columns and then extended columns to a height of $6$.

\[
C=\begin{array}{l}\tableau{{5}\\{0}\\{\overline{8}}\\{\overline{5}}} \end{array} \!\begin{array}{c}  (lC,rC)= \end{array}\! \begin{array}{ll}\tableau{{ 4}\\{7}\\{ \overline{8}} \\ { \overline{5}}}\tableau{{ 5}\\{ \overline{8}}\\{ \overline{7}} \\ { \overline{4}}}  \end{array}\begin{array}{c}(\widehat{lC},\widehat{rC})= \end{array} \begin{array}{ll}\tableau{{ 4}\\{7}\\{ \overline{8}} \\ { \overline{5}}\\ { \overline{2}}\\ { \overline{1}}}\tableau{{1}\\{2}\\{ 5}\\{ \overline{8}}\\{ \overline{7}} \\ { \overline{4}}}  \end{array} 
\]

\end{example}

 Recall the type $B_n$  quantum Bruhat graph criterion (cf. Proposition~\ref{type B bruhat conditions}). The main differences from the type $C_n$ QBG are the loss of the ability to change negative to positive entries with the stage II roots, $(i,\overline{\imath})$, but gaining the ability to change negative to positive entries with the stage IV roots, $(i,\overline{\jmath})$. At first glance, this does not seem to hinder the path algorithm $Path\_C$ from Section~\ref{inverse map type C section}: if such a sign change is necessary, it is merely postponed.  However, while the $(i,\overline{\imath})$ root only changes the sign in position $i$, the $(i,\overline{\jmath})$ root changes the sign in position $j$ as well.  The subtle difference in the quantum Bruhat criterion makes both the reorder and path algorithms from type $C_n$ fail in type $B_n$.  We discuss two modifications to these algorithms, which depend on a certain pattern avoidance in two adjacent columns.

\begin{remark}\label{reason for mod remark} {\rm In types $A_{n-1}$ and $C_n$, the algorithm for forming the correct sequence of roots followed the rule that for the current word, $w$, and the root $(i,j)$, if $w(i)\prec wt_{ij}(i)\preceq C'(i)$, then add the root to the sequence, and otherwise do not and proceed.  We will call a transposition following the above inequality a \textit{Path\_C transposition}.  Further, if $w(i)\prec C'(i)\prec wt_{ij}(i)$, we will say that the transposition \textit{passes the target} (in row $i$). The original path algorithm has two underlying rules: one, we never apply a root which forces us to pass the target, and two, we always use a $Path\_C$ root.  In type $B_n$ (and, as we will later see, in type $D_n$ as well) there are exceptions to these two rules. Both exceptions come directly from the need to avoid the following pattern in two adjacent columns. }
\end{remark}

\begin{definition}\label{block-off def} We say that columns $C = (l_1,l_2,...,l_k)$ and $C' = (r_1,r_2,...,r_k)$ are {\rm blocked off at $i$ by $b:=r_i$} if and only if the following hold:
\begin{enumerate}
\item $ |l_i| \leq b < n$, where $|l_i| = b$ if and only if $l_i = \overline{b}$;
\item $\{1,2,...,b\}\subset \{|l_1|,|l_2|,...,|l_i|\}$ and $\{1,2,...,b\}\subset \{|r_1|,|r_2|,...,|r_i|\}$;
\item $|\{j : 1\leq j\leq i, l_j<0, r_j> 0\}|$ is odd.
\end{enumerate} 
\end{definition}

\begin{example}
 The following columns $CC'$ of height 5 with entries from $[\overline{8}]$ are blocked-off at 4 by 3:

$$
\begin{array}{l}\tableau{{1}&{1}\\{4}&{5}\\{\overline{2}}&{\overline{2}}\\{\overline{3}}&{3}\\{5}&{8}} \end{array}
$$
\end{example}

\vspace{12pt} We will find that if two columns (of length $k$) are blocked off at any $i\in [k-1]$, there will be no corresponding path in the quantum Bruhat graph between the columns. Therefore, given two columns $C$ and $C'$, we must not only avoid this pattern in the reordering of column $C'$, but also make sure that at any point in the aplication of roots in $T$, we never force the current column to be bocked off with $C'$ at any $i\in [k-1]$ either. The latter part will give way to the $Path\_B$ algorithm.

We now define the type $B_n$ versions of the {\it reorder} and {\it path} algorithms. Let $b:=b^l_1b^r_1\ldots b^l_{\lambda_1}b^r_{\lambda_1}=b_1\ldots b_{2\lambda_1}$ be extended split columns indexing a vertex of the crystal $B^{\lambda'}$ of type $B_n$. Similarly, let $\Gamma:=\Gamma^l_1\Gamma^r_1\ldots \Gamma^l_{\lambda_1}\Gamma^r_{\lambda_1}=\Gamma_1\ldots \Gamma_{2\lambda_1}$.  Algorithm~\ref{Mod-Reorder algorithm} takes $b,\Gamma$ as input and returns a reordered filling $C$ of a Young diagram of shape $2\lambda$.

\begin{algorithm}\label{Mod-Reorder algorithm}
(``ord\_B'')

 let $C_1:=b_1$;

 \hspace{6pt}for $i$ from $2$ to $2\lambda_1$ do

 \hspace{12pt} for $j$ from $1$ to $\lambda'_i-1$ do

 \hspace{18pt} let $C_i(j):=min_{\prec_{C_{i-1}(j)}}(b_i\setminus \{C_i(1),\ldots,C_i(j-1)\}$ so that $C_{i-1},C_{i}$ not blocked off at $j$)

 \hspace{12pt} end do;
 
 \hspace{12pt} let $C_i(\lambda'_i) := min_{\prec_{C_{i-1}(j)}}(b_i\setminus \{C_i(1),\ldots,C_i(\lambda'_i-1)\}$

 \hspace{6pt} end do;

return $C:=C_1\ldots C_{2\lambda_1}=C^l_1C^r_1\ldots C^l_{\lambda_1}C^r_{\lambda_1}.$
\end{algorithm}

\begin{example}{\rm 
Algorithm~\ref{Mod-Reorder algorithm} gives the filling $C$ from $b$ below. Note that Algorithm~\ref{Reorder algorithm} would have paired the $3$ with the $\overline{3}$ in the $4^{\rm{th}}$ row.  However, this would cause the two columns to be blocked off at $4$ by $3$, so the modified algorithm skips to the next value and pairs the $8$ with the $\overline{3}$ instead.

 $$b= \tableau{{1}\\{4}\\{\overline{2}}\\{\overline{3}}\\{5}}  \tableau{{1}\\{3}\\{5}\\{8}\\{\overline{2}}} \xrightarrow{ord\_B} \tableau{{1}\\{4}\\{\overline{2}}\\{\overline{3}}\\{5}}  \tableau{{1}\\{5}\\{\overline{2}}\\{8}\\{3}} =C$$
}\end{example}

The $``Path\_B''$ algorithm (Algorithm~\ref{Mod-Greedy algorithm}) takes the reordered, extended, split filling $C=C_1\ldots C_{2\lambda_1}$ given by Algorithm~\ref{Mod-Reorder algorithm}, and outputs a sequence of roots  \textit{Path\_B}$(C) = S\subset \Gamma$. We define $C_0$ to be the increasing column filled with $1,2,\ldots,n$.  Note that the major difference between the $Path\_B$ and $Path\_C$ algorithms is the addition of blocked-off avoidance.

\begin{algorithm}\label{Mod-Greedy algorithm}
(``Path\_B'')
 
  for $i$ from $1$ to $2\lambda_1$ do
 
 \hspace{6pt} let $S_i:=\emptyset$, $A := C_{i-1}$; 

 \hspace{6pt} for $(l,m)$ in $\Gamma_i$ do
 
 \hspace{12pt} if $(l,m)=(i,i+1)$ and $A,C_i$ are blocked off at $i$ by $C_i(i)$, then let $S_i:=S_i,(i,i+1)$, $A:=A(i,i+1)$;

 \hspace{12pt} elsif $A(l)\neq C_i(l)$ and $A(l)\prec A(m)\prec C_i(l)$ and $A(l,m),C_i$ not blocked off at $l$ by $C_i(l)$, then let $S_i:=S_i,(l,m)$, $A:=A(l,m)$;
 
 \hspace{12pt} end if;

 \hspace{6pt} end do;
 
 end do;

 return $S := S_1\ldots S_{2\lambda_1}=S_1^lS_1^r\ldots S_{\lambda_1}^l S_{\lambda_1}^r$.
\end{algorithm}

\begin{example}{\rm 
Consider the crystal $B^{(2,2)}$ of type $B_3$. Then $\lambda' = \lambda = (2,2)$ and $\Gamma = \Gamma(2)\Gamma(2)$.  Suppose that we have $\widehat{rC_1}=\tableau{{\overline{3}}\\{\overline{2}}\\{1}}$ and $\widehat{lC_2}=\tableau{{1}\\{3}\\{2}}$.  Algorithm~\ref{Mod-Greedy algorithm}  produces the following subset of $\Gamma^l(2)=\{(2,3),(2,\overline{2}),(2,\overline{3}),(2,\overline{1}),(1,3),(1,\overline{1}),(1,\overline{\vspace{-1.1mm}\vspace{-1.1mm}3})\}$: 
$$\begin{array}{l}\tableau{{\overline{3}}\\{\overline{2}}} \\ \\ \tableau{{ 1}} \end{array} \!
\begin{array}{c} \\ \xrightarrow{(2,3)} \end{array}\! 
\begin{array}{l}\tableau{{ \overline{3}}\\{ 1} \\ \\ {  \overline{2}}} \end{array} \begin{array}{c} \\ {\xrightarrow{(2,\overline{3})}} \end{array}\!
\begin{array}{l}\tableau{{ \overline{3}}\\{ 2} \\ \\ {  \overline{1}}} \end{array}
\begin{array}{c} \\ {\xrightarrow{(2,\overline{1})}} \end{array}\!
\begin{array}{l}\tableau{{ \overline{2}}\\{ 3} \\ \\ {  \overline{1}}} \end{array} 
\begin{array}{c} \\ {\xrightarrow{(1,\overline{3})}} \end{array}\!
\begin{array}{l}\tableau{{1}\\{ 3} \\ \\ { 2}} \end{array} 
  \, .$$
  
Notice that Algorithm~\ref{Greedy algorithm} would have called for the use of $(1,3)$ instead of $(1,\overline{3})$.  This would have caused the resulting word to be blocked off with $\widehat{lC_2}$ at $1$ by $1$ and thus the natural extension of algorithm $Path\_C$ to type $B_n$ would not have terminated correctly.
  }\end{example}

\begin{theorem}\label{main theorem}
The map ``$\mbox{path\_B}\circ \mbox{ord\_B}\circ \mbox{extend}\circ \mbox{split\_B}$'' is the inverse of the type $B_n$  $``sfill\_B''$ map.
\end{theorem}

The proof of Theorem~\ref{main theorem} is based off of the following proposition.  First we consider the following conditions on a pair of adjacent columns $CC'$ in a filling of a Young diagram.

\begin{conditions}\label{nec conditions for B columns}

For any pair of indices $1 \leq i < l \leq k$,
\begin{enumerate}

\item  $C(i)\neq C'(l)$

\item  $ C(i) \prec C'(l) \prec C'(i)$  only if the columns $C$ and $C't_{il}$ are blocked off at $i$ by $C'(l)$

\item $CC'$ are not blocked off at $i$ by $C'[i]$ for any $1\leq i <k$
\end{enumerate}
\end{conditions}

\begin{remark}\label{main theorem remark}
{\rm Let $b$ be a filling in $B^{\lambda'}$, represented with split columns (as a filling of the shape $2\lambda$).  Then Algorithm~\ref{Mod-Reorder algorithm} constructs the unique filling $\sigma$ satisfying the following conditions: (i) the first column of $\sigma$ (of height $\lambda'_1$) is increasing, (ii) the adjacent columns of $\sigma$ satisfy Conditions~\ref{nec conditions for B columns}, and (iii) $sort(\sigma)=b$.}

\end{remark}

\begin{proposition}\label{main prop}
The restriction of the filling map $``fill\_B''$ to  $\mathcal{A}(\lambda)$ is injective.  The image of this restriction consists of all fillings   $C_1^l C_1^r \hdots C_l^{\lambda_1} C_r^{\lambda_1}$  of the shape $2\lambda$ with integers in $[\overline{n}]$ satisfying the following conditions: $(i)$ $C_l^1$ is increasing; $(ii)$ $(ord\_B(C_l^j),ord\_B(C_r^j))=(lD^j,rD^j)$ for some $KN$ column $D^j$, for all $j$; $(iii)$ any two adjacent columns satisfy  Conditions~\ref{nec conditions for B columns}.
\end{proposition}

The proof of Proposition~\ref{main prop} is based on the following two results, which will be proved in Sections~\ref{R2nextL} and~\ref{L2R} respectively.

\vspace{12pt} \begin{proposition}\label{Total Path Prop}  
Consider a signed permutation $u$ in $B_n$, the column $C:=u[1,k]$, and another column $C'$ of height $k$.  The pair of columns $C'C$ satisfies Conditions~\ref{nec conditions for B columns} if and only if there is a path $u = u_0,u_1,\hdots,u_p=v$ in the corresponding quantum Bruhat graph such that $v[1,k] = C'$ and the edge labels form a subsequence of $\Gamma_l(k)$.  Moreover, the mentioned path is unique, and we have $$C(i) = u_0(i)\preceq u_1(i)\preceq\hdots\preceq u_p(i) = C'(i), \text{ for } i = 1,\hdots,k.$$

\end{proposition}

To prove Proposition~\ref{Total Path Prop}, we will first determine the necessary conditions for a path in the quantum Bruhat graph to exist and then construct a path using a more detailed version of Algorithm~\ref{Mod-Greedy algorithm}.

\vspace{12pt} \begin{proposition}\label{SER prop}

Consider a signed permutation $u$ in $B_n$, the column $C:=u[1,k]$, and another column $C'$ of height $k$.  The pair of columns $CC'$ satisfy Conditions~\ref{nec conditions for B columns} and $(mod\_ord(C),mod\_ord(C')) = (rD,lD)$ for some $KN$ column $D$ if and only if there is a path $u=u_0,u_1,\hdots, u_p = v$ in the corresponding quantum Bruhat graph such that $v[1,k]=C'$ and the edge labels form a subsequence of $\Gamma_r(k)$.  Moreover, the mentioned path is unique and, for each $i=1,\hdots, k$, we have the following weakly increasing sequence: 
$$C(i) = u_0(i)\preceq u_1(i)\preceq \hdots \preceq u_p(i) = C'(i).$$
\end{proposition}

To prove Proposition~\ref{SER prop}, we will first classify the split, extended and reordered KN columns, and then show that the resulting columns are exactly the necessary conditions for constructing a path in the quantum Bruhat graph.

\begin{proof}(of Proposition~\ref{main prop})
Consider a filling $\sigma = C_l^1C_r^1\hdots C_l^{\lambda'_1}C_r^{\lambda'_1}$ satisfying conditions (i)-(iii), and let $u$ be the identity permutation.  Apply Proposition~\ref{Total Path Prop} to $u$ and the culumn $C_l^1$; then set $u$ to the output signed permutation $v$, and apply Proposition~\ref{SER prop} to $u$ and the column $C_r^1$.  Continue in the way, by considering the columns $C_l^2,C_r^2,\hdots ,C_l^{\lambda'_1},C_r^{\lambda'_1}$, while each time setting $u$ to be the previously obtained $v$.  This procedure constructs the unique folding pair $(w,T)$ mapped to $\sigma$ by the filling map $fill\_B$. Viceversa, a similar application of the two propositions shows that any filling in the image of $fill\_B$ satisfies conditions (i)-(iii).

\end{proof}

\begin{proof}(of Theorem~\ref{main theorem})
This in now immediate, based on Remark~\ref{main theorem remark} and Proposition~\ref{main prop}.
\end{proof}

\section{Proof of Proposition~\ref{Total Path Prop}}\label{R2nextL}
\subsection{Necessary conditions for reordered columns}

Let $C$ and $C'$ be two columns of height $k$ with entries from $[\overline{n}]$ where each column cannot contain both $i$ and $\overline{\imath}$ for $1\leq i\leq n$.  Recall Algorithm~\ref{Mod-Reorder algorithm} (the type $B_n$ reorder algorithm) and note that for each such pair of columns, it reorders the second based on the first in a way that is consistant with Conditions~\ref{nec conditions for B columns}.  In this section, we show that the first two parts of Conditions~\ref{nec conditions for B columns} are necessary for the existance of a path between the two columns in the quantum Bruhat graph.
The necessity for the third condition will be realized in Section~\ref{constructing segment of path in B}.

We begin by recalling Conditions~\ref{nec conditions for B columns} on adjacent columns $C C'$ where we now write the second condition in an equivalent way which will be more beneficial in this section.

\vspace{12pt}

For any pair of indices $1 \leq i < l \leq k$,
\begin{enumerate}

\item  $C(i)\neq C'(l)$

\item the statement $ C(i) \prec C'(l) \prec C'(i)$ is false unless the following are true.

\begin{enumerate}
\item $\{1,2,\hdots,|C'(l)|\} \subseteq \{|C(j)|\}_{1\leq j \leq i}$ and $\{1,2,\hdots,|C'(l)|-1\} \subseteq \{|C'(j)|\}_{1\leq j \leq i-1}$
\item $ C'(l)>0$ and $|C(i)|\leq C'(l)$ with equality iff $C(i) = \overline{C'(l)}$,
\item and $p'$ is odd, where $p=\#\{j : 1\leq j\leq i-1, C(j)<0, C'(j)> 0\}$ and $p' = p+1$ if $C(i)<0$ and $p' = p$ otherwise.
\end{enumerate} 

\item $C,C'$ are not blocked off at $i$ by $C'[i]$ for any $1\leq i <k$
\end{enumerate}


\vspace{12pt} The following lemma will show that for a given path in the quantum Bruhat graph, the only possible transposition in $T_i$ that can pass the target value $C'(i)$ is $(k,k+1)$.  It also gives us the monotonicity of the values in position $i$ during the application of  transpositions of $T_i$.

\begin{lemma}\label{row monotonicity}  
For $T_i = (i,j_1)(i,j_2)\hdots (i,j_q)$, some $1\leq i \leq k$, let $u_0:=\pi_{i+1}$ and $ u_r := \pi_{i+1}t_{i,j_1}t_{i,j_2}\hdots t_{i,j_r}$ for $1\leq r\leq q$. Then $u_r(i)\prec u_{r+1}(i)\preceq C'(i)$ for $r=1,\hdots, q-1$, while for $r=0$ this fails only if $(i,j_1) = (k,k+1)$.
\end{lemma}



\begin{proof}
Suppose that there is some $0\leq r<q-1$ such that $u_r(i)\prec C'(i) \prec u_{r+1}(i)$.  We first show that this cannot happen during the application of transpositions in $T_i$ for $i<k$ and then we show that it can only happen for $j_1 = k+1$ in $T_k$.

Let $i<k$.  Then if we have $u_r(i)\prec C'(i) \prec u_{r+1}(i)$, the QBG criterion then gives that \[u_r(i)\prec C'(i) \prec u_{r+1}(i)\prec u_r(k) = C'(k)\prec u_r(i).\]  Therefore some transposition $(i,j_{r'})$, for $r+1 < r' \leq q$ will transpose over the value $C'(k)$, breaking the QBG criterion.

Now let $i=k$. Consider the case where $j_{r+1}\neq k+1$.  We will show that $u_r(i)\prec C'(i) \preceq u_{r+1}(i)$ cannot occur and the result follows. We break into two cases:  when $j_{r+1} \neq \overline{k+1}$ and when $j_{r+1} = \overline{k+1}$.

\underline{Case 1:} consider when $j_{r+1} \neq \overline{k+1}$.  Then by the QBG criterion, we have \[u_r(k)\prec C'(k)\prec u_{r+1}(k) \prec u_r(k+1)\prec u_r(k).\] However, the only way to transpose over the value in position $k+1$ is with the root $(k,\overline{k+1})$, after which there is no longer a root which can transpose over the root in position $j_r$.

\underline{Case 2:} consider when $j_{r+1} = \overline{k+1}$.  Then there is still the conflict of transposing over the value in position $j_r$.  Indeed, this could only be done with a Stage III transposition, of which $(k,\overline{k+1})$ is the last.

Therefore the only possible transposition where passing the target value $C'(i)$ can possibly occur is with the root $(k,k+1)$. 
\end{proof}

\vspace{12pt} The next lemma shows us that if we have two columns $CC'$ satisfying Condition 2, we know exactly where in the application of roots in $T$ the current word's value in position $i$ becomes greater than $C'(l)$ in the circle order starting at $C(i)$.

\begin{lemma}\label{Lenart reorder lemma} Let $CC'$ be a pair of columns such that $a:=C(i)\prec b:=C'(l)  \prec c:=C'(i)$ for some $1\leq i < l \leq k$. Then the reflection $(l,\overline{\imath})$ is in $T_l$. Furthermore, if $w$ is the word immediately prior to the application of $(l,\overline{i})$, then we have that $w(i)\prec C'(l)\prec C'(i)$ and $C'(l)\prec wt_{l,\overline{\imath}}(i)\prec C'(i)$.
\end{lemma}

\begin{proof}

 Let us assume that we have $a:=C(i)\prec b:=C'(l)  \prec c:=C'(i)$ for some $1\leq i < l \leq k$.  We first show that there is a root with the said properties and then show that this root is exactly the root $(l,\overline{\imath})$.

  First, note that we cannot have $C_l(i)\preceq b\prec c$, otherwise the entry in position $i$ of the signed permutation would then change from $C_l(i)$ to $C'(i)$ via reflections in $T_{l-1}T_{l-2}\hdots T_i$. One of these reflections would then transpose entries across $b$, violating the quantum Bruhat graph criterion. It must then be that $b \prec C_l(i)\preceq c$.  This means that at some point in the process of applying to $u$ the reflections in $T_k,T_{k-1},\hdots , T_l$, we apply to the current signed permutation $w$ a reflection $(i,\overline{m})$ such that $a\prec a':= w(i)\prec b\prec c':= w(\overline{m})\prec c$.  Let $(i_0,\overline{m}),\hdots , (i_p,\overline{m})$ be the segment of $T_m$ starting with $(i,\overline{m})$ and consisting of roots $(- , \overline{m})$, where $i = i_0 > i_1 > \hdots >i_p\geq 1$.  Let $a_r:=w(i_r)$ for $0\leq r\leq p$.

We now show that $(i,\overline{m})$ is our desired root, or equivalently, that $m=l$. Since  $wt_{i,\overline{m}}t_{i_1,\overline{m}}\hdots t_{i_j,\overline{m}}=\pi_m$, the quantum Bruhat criterion gives that $\pi_m[i_p,\overline{m}]$ cannot have any entries between $a_p$ and $c'$ except maybe $\overline{a_p}$ in the $m^{th}$ position of $C_m$. It now follows that since $i_p\leq i<l\leq m <\overline{m}$, we cannot have $a_p\prec C_m(l)\prec c'$ unless $m = l$ where $C_m(m) = \overline{a_p}$. Suppose that $m\neq l$. Then $C_m(l)\prec a_p\prec c'$. Note that by Lemma~\ref{row monotonicity}, $\overline{c'}\prec \overline{a'}=\overline{a_0}\prec \overline{a_1}\prec\hdots\prec \overline{a_p}$. This gives us that $a_p \prec a' \prec c'$ and so $C_m(l)\prec a_p \prec a'\prec b \prec c'$. Since $C'(l) = b$, there would then have to be some transposition in $T_{m-1},\hdots, T_1$ which would transpose two values over $a_p$ in position $\overline{m}$. This contradicts the quantum Bruhat criterion, and so it must have been that $\overline{a_p} = b$ and $m=l$.
\end{proof}

\begin{example}\label{Lenart reorder lemma example}
{\rm Consider the following quantum Bruhat graph path in $B_8$. Here we have $u[1,6]$ and $v[1,6]$ follow the proposed reorder criterion. Note that $u[4]\prec v[6]\prec v[4]$ and that $u[1,6]$, $vt_{4,6}[1,6]$ is blocked off at $4$ by $3$. We see not only that $t_{6,\overline{4}}$ is indeed in the path, but also that while $w[4]\prec v[6]\prec v[4]$ the next transposition gives $wt_{6,\overline{4}}[4]\prec v[4]\prec v[6]$ .

$$\begin{array}{l}{u} \\ \tableau{{1}\\{4}\\{ \overline{2}}\\{\overline{3}}\\{8}\\{7}}
 \\ \\ \tableau{{5}\\{6}} \end{array} 
 \!\begin{array}{c} \\ \xrightarrow{(6,\overline{6})} \end{array}\!\begin{array}{l} \\ \tableau{{1}\\{4}\\{ \overline{2}}\\{\overline{3}}\\{8}\\{\overline{7}}}
 \\ \\ \tableau{{5}\\{6}} \end{array} 
 \!\begin{array}{c} \\ \xrightarrow{(6,\overline{8})} \end{array}\! \begin{array}{l} \\ \tableau{{1}\\{4}\\{ \overline{2}}\\{\overline{3}}\\{8}\\{\overline{6}}}
 \\ \\ \tableau{{5}\\{7}} \end{array} 
 \!\begin{array}{c} \\ \xrightarrow{(6,\overline{7})} \end{array}\! \begin{array}{l} {w} \\ \tableau{{1}\\{4}\\{ \overline{2}}\\{\overline{3}}\\{8}\\{\overline{5}}}
 \\ \\ \tableau{{6}\\{7}} \end{array} 
 \!\begin{array}{c} \\ \xrightarrow{(6,\overline{4})} \end{array}\! \begin{array}{l} {wt_{6\overline{4}}} \\ \tableau{{1}\\{4}\\{ \overline{2}}\\{5}\\{8}\\{3}}
 \\ \\ \tableau{{6}\\{7}} \end{array} 
 \!\begin{array}{c} \\ \xrightarrow{(4,7)} \end{array}\!  \begin{array}{l} \\ \tableau{{1}\\{4}\\{ \overline{2}}\\{6}\\{8}\\{3}}
 \\ \\ \tableau{{5}\\{7}} \end{array} 
 \!\begin{array}{c} \\ \xrightarrow{(2,7)} \end{array}\!  \begin{array}{l} {v}\\ \tableau{{1}\\{5}\\{ \overline{2}}\\{6}\\{8}\\{3}}
 \\ \\ \tableau{{4}\\{7}} \end{array} 
 \!
  \,$$
}
\end{example}

\begin{remark}\label{C_m contains 1,...,b}
{\rm Let $CC'$ be a pair of columns such that $a:=C(i)\prec b:=C'(l)  \prec c:=C'(i)$ for some $1\leq i < l \leq k$. We claim that $\{1,2,\hdots,|b|-1\}\subseteq\{|C_m[j]|\}_{1\leq j\leq i}$ and $\{1,2,\hdots,|b|-1\}\subseteq\{|C_{m+1}[j]|\}_{1\leq j\leq i}$.  For the first, recall from the proof of Lemma~\ref{Lenart reorder lemma} that there cannot be any values between $\overline{b}$ and $c'$ in $\pi_m[i,\overline{m}]$.  The same proof gave way to inequalities which tell us that $|b|<|c'|$.  This means that there cannot be any values between $\overline{b}$ and $b$ in $\pi_m[i,\overline{m}]$.  Since $i<m$, we have that there cannot be any values between $\overline{b}$ and $b$ in $\pi_m[i,\overline{i}]$. The first claim then follows.  For the second, notice that the only difference between the content of $C_{m+1}[1,i]$ and $C_m[i,i-1]$ is the loss of $a_p=\overline{b}$.}
\end{remark}

\vspace{12pt}
The following Lemma will be used for showing the necesity of Condition 2c.

\begin{lemma}\label{pos order} Suppose that there is $1\leq i < l \leq k$ such that $C(i)\prec C'(l)\prec C'(i)$, and Conditions 2a,b are met. Then for all $1\leq j\leq i-1$, if $C(j),C'(j)>0$, then $C(j) < C'(j)$.
\end{lemma}

\begin{proof}
 Suppose that for such a $j$, there is an $r\in\{1,\hdots,q\}$ such that $u_r(j)<0$. Recall that by the quantum Bruhat criterion, all upsteps in stages III and IV must maintain the same sign.  This means that the transposition acting on the current word $w$ changing position $j$ from positive to negative must come from a stage I or II reflection which can only happen in $T_j$ at $(j,s_0)$ for some $s_0\in\{k+1,\hdots,\overline{k+1}\}\cup\{\overline{j}\}$. Since Condition 2a holds, $wt_{j,s_0}(j)\notin [\overline{b}]$. Let the remaining stage I reflections be $(j,s_1),\hdots,(j,s_p)$ where $k+1\leq s_0 < \hdots <s_p\leq n$. Notice that $w[s_i]\notin [\overline{b}]$ for any $0\leq i\leq p$ by Condition 2a. This means that $b\prec w'(j)=wt_{j,s_0}\hdots t_{j,s_p}[j]\prec \overline{b}$. However, if $w'[j]\in \{b,\hdots,n\}$, then one of the reflections $(j,s_i)$ would have to transpose two values over $b$ in position $l$, contradicting the quantum Bruhat criterion.  Thus $w'(j)\in\{\overline{n},\hdots,\overline{b+1}\}$. Note that we cannot apply the $(j,\overline{j})$ reflection at this time by quantum bruhat criterion, so we move on to stage III reflections. Again by Condition 2b, each stage III reflection can only change position $j$ to a value in $\{b+1,\hdots,\overline{b+1}\}$. However, no stage III reflection can take position $j$ to anything in $\{b+1,\hdots,w'[j]\}$ without contradicting either the quantum Bruhat criterion or Lemma~\ref{row monotonicity}.  Thus if $w''$ is the word at the end of applying stage III transpositions, we have that $w''(j)\in\{\overline{n},
\hdots,\overline{b+1}\}$. This means that there must be a stage IV reflection taking $w''(j)$ to a positive value.  This cannot happen, as it would require a reflection to transpose over $\overline{b}$ in position $\overline{l}$, contradicting the quantum Bruhat graph criterion.  Thus there is no such $u_r(j)<0$ and the lemma holds.

\end{proof}  

\begin{remark}\label{One Downstep per row}
{\rm By Lemma~\ref{row monotonicity}, it is clear that for each $1\leq j < k$, there may only be at most one downstep during $T_j$. Otherwise we would inevitably pass over the target.
}
\end{remark}

\vspace{12pt} The following Lemma was given by Lenart in~\cite{Lenart 2012} and is not type dependant so we will use it freely.

\begin{lemma}\label{Lenart 7.1}
Fix $i$ and $l$ with $1\le i < l \leq k$. Let $a$ be the entry in position $i$ of the signed permutation obtained at some point in the process of applying to $u$ the reflections in $T_k,T_{k-1},\hdots, T_{l+1}$. Then either $a$ appears in $C_{l+1}[1,i]$ or $\overline{a}$ appears in $C_{l+1}[l+1,k]$.

\end{lemma}

\begin{proposition}\label{Reorder Nec}The pair of columns $CC'$ satisfy Conditions $1$ and $2$.\end{proposition}

\begin{proof} The proof of condition 1 is the same as in type $C_n$.  Let us now assume that we have $a:=C(i)\prec b:=C'(l)  \prec c:=C'(i)$ for some $1\leq i < l \leq k$.  Let $a',c',m,$ and $w$ be as in Lemma~\ref{Lenart reorder lemma}.

\vspace{12pt} We first show that $\{1,\hdots,b-1\}\subseteq\{|C'[j]|\}_{1\leq j\leq i}$. It is equivalent to show that values in $[\overline{b-1}]$ are not in $\pi_1[i+1,n]$.  This is true for positions $m$ though $k$, as $\pi_1[m,k] = \pi_m[m,k]$ which was shown in Remark~\ref{C_m contains 1,...,b} to not contain elements of $[\overline{b-1}]$.  Suppose that $|\pi_1(r)|\in\{1,\hdots,b-1\} $ for some $r\in\{i,\hdots,m-1\}$.  Then (from the proof of Lemma~\ref{Lenart reorder lemma}) $\pi_m(r)$ is between $c'$ and $\overline{b}$ and  $C'(r)=\pi_1(r)$ is in either $\pi_m[1,i-1]$ or $\pi_m[\overline{i-1},\overline{1}]$ and by Lemma~\ref{Lenart 7.1}, necessarily the latter.  But during the reflections in $T_r,\hdots,T_m$ there would have to be a reflection that transposes over $\overline{b}$ in position $\overline{m}$, contradicting the quantum Bruhat criterion.  Now suppose that $|\pi_1(r)|\in[\overline{b-1}]$ for some $r\in\{k,\hdots,n,\overline{n},\hdots,\overline{k+1}\}$. But this would mean that some reflection in $T_{m-1},\hdots,T_1$ would have to transpose an entry over $b$ in position $m$, contradicting the the quantum Bruhat criterion.  Thus there are no elements of $[\overline{b-1}]$ in $\pi_1[i+1,\overline{i+1}]$ and so $\{1,\hdots,b-1\}\subseteq\{|C'[j]|\}_{1\leq j\leq i}$ as desired.

\vspace{12pt}
 We now show that $\{1,\hdots,b\}\subseteq\{|C[j]|\}_{1\leq j\leq i}$. If not, then, given what we know about the structure of $\pi_{m+1}$ from Remark~\ref{C_m contains 1,...,b}, there must be a $t\in\{m+1,\hdots,k\}$ and $i_1'\in\{1,\hdots,i\}$ where the application of the reflection $(i_1',\overline{t})$ to the current word $w'$ transposes value $\overline{\gamma}=w'[\overline{t}]$ and $w'[i_1']$ where $\overline{b}\preceq\gamma\preceq b$.
     Let $(i_1',\overline{t}),\hdots , (i_r',\overline{t})$ be the remainder of $T_t$ starting with $(i_1',\overline{t})$ and consisting of roots $(\cdot , \overline{t})$, where $ i_1' > i_2' > \hdots >i_r'\geq 1$.

     We claim that  $\overline{\mu}:=w't_{i_1',\overline{t}}\hdots t_{i_{r-1}',\overline{t}}[\overline{i_r'}]=w't_{i_1',\overline{t}}\hdots t_{i_r',\overline{t}}[t]=C_t[t]=C'[t]\in\{c',\hdots,\overline{c'}\}$, $c'>0$, and further that no values in $w'[i,\overline{t}]$ can lie between $\mu$ and $\overline{\gamma}$. This means that $w'[m]$ lies between $\overline{\gamma}$ and $\mu$.  But, from the construction in the proof in Lemma~\ref{Lenart reorder lemma}, we had that $C_m[m] = \overline{c'}$, which would be impossible, as we would then have to traverse over $\overline{\mu}$ in position $t$ during one of the reflections in $T_{t-1},\hdots,T_{m+1}$, breaking the quantum Bruhat criterion.  Therefore there was no such $\gamma$ and it must be that $\{1,\hdots,b\}\subseteq\{|C[j]|\}_{1\leq j\leq i}$.
     
     We finish this part by proving the two claims.  First, that $\overline{\mu}:=w't_{i_1',\overline{t}}\hdots t_{i_{r-1}',\overline{t}}[\overline{i_r'}]=w't_{i_1',\overline{t}}\hdots t_{i_r',\overline{t}}[t]=C_t[t]=C'[t]\in\{c',\hdots,\overline{c'}\}$. Recall from construction in the proof of Lemma~\ref{Lenart reorder lemma} that the reflection $(m,\overline{i})$ will be used to transpose $a'$ and $c'$ in $w$ during the application of reflections in $T_m$ and so there are no values between $a'$ and $c'$ in $w[i,\overline{m}]$ and nothing betwen $\overline{c'}$ and $\overline{a'}$ in $w[m,\overline{i}]$. This means that all values in $w[m,\overline{m}]$ lie in $\{\overline{a'},\hdots,a'\}\cup\{c',\hdots,\overline{c'}\}$.  Notice that $w[m+1,\overline{m+1}]=\pi_m[m+1,\overline{m+1}]=\pi_1[m+1,\overline{m+1}]$ and since $\overline{b}\leq a'<b$, we have that $a'\in [\overline{b}]$, so by the condition set on the position of the values on $\pi_m$ in Remark~\ref{C_m contains 1,...,b}, we get that values in $w[m+1,\overline{m+1}]$ must come from $\{c',\hdots,\overline{c'}\}$.  Note that this forces $c'>0$. The claim follows, as $t\in\{m+1,\hdots,k\}$. 
     
     For the second part of the claim, notice that by the quantum Bruhat criterion, no values in $w'[i_1',\overline{t}]$ can lie between $w'[i_1']$ and $w'[\overline{t}]=\overline{\gamma}$. Similarly, no values in $w't_{t,\overline{i_1'}}[i_2',\overline{t}]$ can lie between $w'[i_2']$ and $w't_{t,\overline{i_1'}}[\overline{t}]=\overline{w'[i_1']}$, so no values in $w'[i,\overline{t}]$ can lie between $w'[i_2']$ and $\overline{\gamma}$. Continue the same way with the rest of the $(t,\cdot)$ transpositions and we see that no values in $w't_{i_1',\overline{t}}\hdots t_{i_r',\overline{t}}[i_r',\overline{t}]$ can lie between $w'[i_r']=\mu\in\{c',\hdots,c\}$ and $\overline{w'[i_{r-1}']}$, so no values in $w'[i,\overline{t}]$ can lie between $\mu$ and $\overline{\gamma}$.
     
     \vspace{12pt} We now wish to show that the existence of $a:=C(i)\prec b:=C'(l)  \prec c:=C'(i)$ implies Condition 2b.  First, suppose that $b<0$. Recall that $\overline{b}\prec a'\prec b\prec c'$.  We have already shown that $c'>0$ and so $c'<a'$ and $\overline{c'}>\overline{a'}$.  Then the reflection $(m,\overline{\imath})$ which transposes $\overline{c'}$ with $\overline{a'}$ is a downstep and, by the quantum Bruhat criterion, it must be that $a'$ and $c'$ are of different sign.  This means that $a'<0$. But then via Lemma~\ref{Lenart reorder lemma}, the reflections $(m,\overline{\imath_1}),\hdots(m,\overline{\imath_p})$ taking $\overline{a'}$ to $b$ in position $m$ must all be upsteps.  The quantum Bruhat criterion then gives that each transposition must maintain the same sign, and so $sign(\overline{a'})=sign(b)$, a contradiction.  Thus $b>0$.
     
     It remains to be shown that $\overline{b}\preceq a\prec b$. We do so by showing that $a=a'$, for which this inequality already holds.  Suppose that $a\neq a'$.  Then $C[i]=a$ by assumption, and $a'=C[j]$ for some $j\in\{1,\hdots,i-1\}\cup\{\overline{i-1},\hdots,\overline{1}\}$ by Condition 2a.  Then by Lemma~\ref{Lenart 7.1}, with $i=j$ and $l=m$, we get that $a'$ must appear in either $C_{m+1}[1,i]$ or $C_{m+1}[m+1,k]$.  But by a claim in the  proof of 2a, we saw that neither $a'$ nor $\overline{a'}$ can appear in the latter, and thus one of them must appear in the prior. Since $j<i$ the transposition $(m,\overline{i})$ occurs before the transposition $(m,\overline{j})$, the transposition $(m,\overline{i})$ cannot transpose the values $\overline{c'}$ with $a'$ as it had originally been constructed, a contradiction.  Thus $a=a'$ and Condition 2b holds.
     
     \vspace{12pt} We now wish to show that the existence of $a:=C(i)\prec b:=C'(l)  \prec c:=C'(i)$ implies implies Condition 2c. We first count the number of downsteps made with transpositions with values in positions $[1,i]$.
     
     First we look at the number of downsteps in $T_k,\hdots, T_{l+1}$.  Here, the downsteps including positions $1,\hdots,i-1$ are only those of stage IV.  By our construction thus far, we know that this transposes a negative value in $\{b+1,\hdots,\overline{b+1}\}$ with a positive value. This then transposes over $\overline{a}$ in position $\overline{i}$, so by the quantum bruhat criterion, we know that no such downstep exists.
     
     Second, we look at the number of downsteps in $T_l$.  By Condition 2b, $C[l]\in\{b+1,\hdots,\overline{b+1}\}$, and by hypothesis, $C'[l]=b$. Thus there must be at least one downstep during $T_l$. From Remark~\ref{C_m contains 1,...,b} we have that $\{1,\hdots,b\}$ are in $\pi_l[1,i]$ or $\pi_l[\overline{i},\overline{1}]$. By Lemma~\ref{Lenart reorder lemma}, the downstep must be a stage IV move with a value in some position in $\{1,\hdots,i\}$.  By Remark~\ref{One Downstep per row} there is at most one downstep in $T_l$, so there is exactly one downstep in $T_l$.  We note here that the value in position $i$ will never become negative again, as the only means of doing so would be with the root $(i,\overline{i})$ in $T_i$, but this would pass the target $b>0$.
     
     Third, we look at the downsteps in $T_{l-1},\hdots, T_{i}$. By similar argumentation to that of $T_k\hdots T_{l+1}$, we see that any possible downstep with values in $1,\hdots, i$ would pass over $b$ in position $l$, breaking the quantum bruhat criterion.

     Finaly, we look at the downsteps in $T_{i-1},\hdots,T_1$.  Notice that there are no downsteps with stages I,III without passing over $b$ in position $l$.  There are no stage II downsteps at all in type $B_n$.  So all downsteps are of stage IV. Each of these change two positions in $1,\hdots i$ from negative to positive. We note here that none of these positions can became negative again. Indeed the only way for this to happen would be with a stage II upstep, but this root has already been passed to use the stage IV downstep.
     
     We have seen that there are an even number of downsteps in positions $[1,i]$ contributed by $T_{i},\hdots,T_1$, one from $T_l$, and none from anywhere else. We noted along the way that none of these values will become negative again during the remainder of the transpositions.  This along with Lemma~\ref{pos order} and Remark~\ref{One Downstep per row} gives that there is an odd number of such negative to positive pairs in the said positions of $C$ and $C'$, as desired.
     
\end{proof}

\subsection{Necessary conditions for the construction of a segment of the quantum Bruhat path}\label{Nec cond for construction}

The following proposition shows a partial result for Proposition~\ref{Total Path Prop}.

\vspace{12pt} \begin{proposition}\label{Path Prop}
 Let $u,i,C,C'$ be as previously defined and assume that the pair of columns $CC'$ satisfies conditions 1 and 2.  Assume also that $C[i+1,k]=C'[i+1,k]$ for some $i$ with $1\leq i\leq k$.  Then there is a unique path $u = u_0,u_1,\hdots,u_q = v$ in the corresponding quantum Bruhat graph such that $v(i) = C'(i)$ and the edge labels form a subsequence of $\Gamma_{ki}$. Moreover, we have that if $u(l)\neq v(l)$, then $C(l)=u(l)\prec v(l)\preceq C'(l)$ for $l= 1,\hdots, i-1$.
\end{proposition}

\vspace{12pt} In this section, we will provide necessary conditions for the results of Proposition~\ref{Path Prop}.  Assume for the moment that a path with the property stated in Proposition~\ref{Path Prop} exists.  Let $T$ be the sequence of edge labels for this path.  Recall that the sequence of roots $\Gamma_{ki}$ were split into four stages.  We will factorize accordingly, giving $T=T_iT_{ii}T_{iii}T_{iv}$ and define $u_i:=uT_i$, $u_{ii}:=uT_iT_{ii}$, $u_{iii}:=uT_iT_{ii}T_{iii}$, and $u_{iv}:=uT=v$.

\vspace{12pt} The following lemmas give necessary conditions for the construction in the Proposition. They show that we will pass the target (cf. Remark~\ref{reason for mod remark}) only with the root $(k,k+1)$ exactly when the reordered columns $CC'$ are blocked off at $k$ and that we will skip a $Path\_C$ transposition exactly when the $Path\_C$ transposition would lead the resulting column to be blocked off with $C'$. We note here that these necessary conditions give the uniqueness of the path proposed in Proposition~\ref{Path Prop}.


\vspace{12pt} Due to the third part of the block off condition, we will often need to discuss the signs of values in the same row of two columns.  For two columns $L$ and $R$, we will denote the signs of position $i$ of these columns by $sgn(l_i)sgn(r_i)$. 

\begin{example}\label{+- example}
{\rm Consider the columns $C = (1,4,\overline{2},\overline{3},8,\overline{5})$ and $C' = (1,5,\overline{2},6,8,3)$. We then say that the values of $CC'$ in position $4$ are -+.  We will commonly discuss how these sign pairs change after applying transpositions to the left word.  We see that $Ct_{6\overline{4}}$ and $C'$ then has a ++ pair in position $4$. 
$$\hspace{10pt} C \hspace{4pt} C' \hspace{20pt} Ct_{6\overline{4}} \hspace{4pt} C'$$
$$\begin{array}{ll} \tableau{{1}&{1}\\{4}&{5}\\{ \overline{2}}&{ \overline{2}}\\{\overline{3}}&{6}\\{8}&{8}\\{\overline{5}}&{3}} \end{array} 
 \! \begin{array}{c} \\ \rightarrow \end{array}\! \begin{array}{ll}\tableau{{1}&{1}\\{4}&{5}\\{ \overline{2}}&{ \overline{2}}\\{5}&{6}\\{8}&{8}\\{3}&{3}} \end{array} 
 \!
  \,$$
  
  }
  \end{example}

We note that we can restate part 3 from Definition~\ref{block-off def} as \textit{there are an odd number of $-+$ pairs in positions 1 through $i$ in the coloumns $CC'$}.
\begin{lemma}\label{no path for block off} Suppose that $u[i+1,k] = C'[i+1,k]$ and $u,C'$ are blocked off at $i$ by $b:=C'(i)$.  Then, if we never pass the target, there is no $T\subset \Gamma_{ki}$ such that $uT[1,k]=C'$.
\end{lemma}

\begin{proof} Let $u$ and $C'$ be as hypothesised.  By the block off condition, either $u(i) = \overline{b}$ or $|u(i)|<b$.  We first consider the case where $u(i) = \overline{b}$.  Since our target value is in position $\overline{\imath}$, the only root available to us are those in stage I and stage II.  However, by the block off condition, there is no $k<l\leq n$ such that $\overline{b}\prec u(l)\prec b$ and the root $(i,\overline{\imath})$ does not follow the quantum Bruhat graph criterion, as $\overline{b}<0$.  Since we do not pass the target, there is no path.

We now consider the case where $|u(i)|<b$. First note that there are an odd number of $-+$ pairs in $u[1,i]$, $C'[1,i]$ due to the block off condition.  We claim that if there is some subchain $T_i\subset \Gamma_{ki}$ such that $uT_i(i)=b$, then there are still an odd number of $-+$ pairs between $uT_i[1,i-1]$ and $C'[1,i-1]$.  Indeed, if $1\leq u(i)<b$, then by Lemma~\ref{row monotonicity}, only upsteps will be used in $T_i$ and so $sgn(u(l)) = sgn(uT_i(l))$ for all $1\leq l <i$.  We also consider the case $\overline{b}\prec u(i)\preceq \overline{1}$.  Here, there must be a downstep at some point in $T_i$.  It can not be in stage I or III, otherwise we would pass the target due to the block off condition.  There are no stage II downsteps by quantum Bruhat criterion.  Thus the downstep must be in stage IV.  This single downstep will not only change the values in position $i$ from $-+$ to $++$, but also a $-+$ to a $++$ in some position $1\leq l < i$.  Therefore, regardless of the sign of $u(i)$, the number of $-+$ values in $uT_i[1,i-1]$ and $C'[1,i-1]$ will be odd.

The conclusion of the proof is done by showing that if there is a path $T\subset rev(\Gamma)$ such that $uT_iT = C'$, there must have been an even number of $-+$ pairs between $uT_i[1,i-1]$ and $C'[1,i-1]$.  We will do so by first showing that no $++$ pairs will give way to $-+$ at any time during $T$ and further that all $-+$ pairs are only changed to $++$ pairs via stage IV downsteps, each of which turn two $-+$ pairs to $++$.  For the remainder of the proof, let $u' = uT_i$ and note that $u'[i,k] = C'[i,k]$ and $u'(i) = b = C'(i)$.

We first consider some $1\leq l <i$ where $u'(l)>0$ and $C'(l)>0$ and will show that position $l$ will never be negative thoughout the application of the remainder of the roots in $T$. First note that by the quantum Bruhat criterion, we have that position $l$ remains positive throughout  $T_{i-1}T_{i-2}\hdots T_{l+1}(l)$. Let $s:=u'T_{i-1}T_{i-2}\hdots T_{l+1}(l)$. Then we have the following two cases: either $1\leq s \leq C'(l)$ or $C'(l)< s\leq n$.  For the first case, we are done via Lemma~\ref{row monotonicity}.  Now suppose that $C'(l)< s \leq n$.    During stages I, II, and III of $T_l$, we can only transpose with values between $b$ and $\overline{b}$, due to the nature of the block off condition.  This means that at some point during stage IV of $T_l$ two values would have to transpose over $\overline{b}$ in position $\overline{\imath}$, breaking the quantum Bruhat criterion.  Therefore only the first case holds and no $++$ pair between $u'[1,i-1]$ and $C'[1,i-1]$ will ever give way to a $-+$ pair during the remainder of $T$.

Finally, we show that all $-+$ pairs in $u'[1,i-1]$ and $C'[1,i-1]$ can only become $++$ pairs via stage IV downsteps, effectively changing two such $-+$ pairs at a time. We now consider some $1\leq l <i$ such that $u'(l)<0$ and $C'(l)>0$ and again let $s = u'T_{i-1}T_{i-2}\hdots T_{l+1}(l)$.  Suppose that during $T_{i-1}T_{i-2}\hdots T_{l+1}(l)$ there are no stage IV down steps using position $l$. We then have that $s$ is negative.  Further, we note that $1\leq C'(l) <b$, otherwise there would come a point where we would need to transpose over $b$ in position $i$. Due to the nature of having been blocked off, nothing between $s'$ and $t$ can be in stages I or III.  As we do not pass the target, none of the roots in those positions will be used.  Further, we can not use a stage II move, as downsteps are not promitted here by the quantum Bruhat criterion.  Therefore there must be a downstep in stage IV.  This concludes the proof.
\end{proof}

\begin{remark} 
{\rm Since we never pass the target except potentially at $(k,k+1)$, for a path to exist, two columns can never be blocked off with each other except for possibly at $u[1,k],C'$.  Further, Lemma~\ref{no path for block off} tells us that there may be times when the $Path\_C$ algorithm would call for a certain transposition, but it can not be included in the path, as its use  would cause the current column to be blocked off with $C'$.  The following lemmas will show that block-off avoidance is the only time a path is allowed to skip a root called by the $Path\_C$ algorithm or pass the target with the root $(k,k+1)$.}
\end{remark}

\begin{lemma}\label{skip implies block off}Let the sequence $m_1,m_2,\hdots , m_r$ be such that $u t_{i,m_1} t_{i,m_2}\hdots t_{i,m_r}(i)=C'(i)$ 
and the roots $(i,m_l)$ form a path in the corresponding quantum Bruhat graph.  Suppose that there is some
\\ $m_d\in \{k,\hdots,\overline{k},\overline{i},\hdots\overline{1}\}$, such that at the current word $w$, we have $w(i)\prec wt_{i,m_d}(i)\preceq C'(i)$, but $m_d$ is not in $\{m_1,m_2,\hdots,m_r\}$.  Then $wt_{i,m_d}$ was blocked off at $i$ by $C'(i)$.
\end{lemma} 

\begin{proof} In the hypothesised construction, at some point in the process of applying the transpositions $(i,m_l)$ to $u$, the entry in position $i$ has to change from a value $a$ to a value $b$ across the value $u(m_d)$ where $a\prec u(m_d)\prec b$.  This violates the quantum Bruhat graph criterion in all but the following two special cases:
 
 Case 1: $a<0$ and $b>0$, the current transposition $(i,m_j)$ is in $T_{iii}$, and $u(m_d)=\overline{b}$.
 
 Case 2: $a<0$ and $b>0$, the current transposition $(i,m_j)$ is in $T_{iv}$, and $u(m_d)=\overline{a}$.
 
 \vspace{12pt} We show that the lemma holds for case one and case two follows similarly. Let $w' := ut_{i,m_1}\hdots t_{i,m_{j-1}}$.  Then $w't_{i,m_j}(i) = b$. We need $b=C'(i)>0$.  Note that after stage II, the quantum Bruhat criterion no longer allows any positive to negative sign changes.  Since $b>0$, it must be that the target value $t = C'(i)$ is positive as well.  Lemma~\ref{row monotonicity} then gives that $1\leq b \leq t \leq n$, providing us with the desired $|u(m_d)|\leq t$ as $u(m_d) = \overline{b}$ by hypothesis. 
 
 We now show that $\{1,\hdots, t-1\}\subset\{|w(j)|\}_{1\leq j\leq i}$ and $\{1,\hdots, t\}\subset\{|C'(j)|\}_{1\leq j\leq i}$.  First note that we cannot have $b < \overline{a} < t$.  Otherwise there would have to be root later in the sequence transposing two values $c$ and $d$ over $\overline{a}$ in position $\overline{m_j}$.  Again, this is only possible if $d = \overline{\overline{a}} = a$, but that would mean that the transposition would pass the target.  Lemma~\ref{row monotonicity} gave that this is only possible for the root $(k,k+1)$, but $m_j$ is in $T_{iii}$ and is therefore not $k+1$. So we now have
  $$\overline{n}\preceq a \preceq \overline{b}\preceq\overline{1}\prec 1\preceq b\preceq t \preceq \overline{a}\preceq n.$$ 
  
  By the quantum Bruhat criterion, there are no values between $a$ and $b$ in $w'[i,m_j]$ except $\overline{b}$ in position $\overline{m_j}=m_d$. Suppose that some value between $b$ and $t$ lies in $w'[i,m_j]$.  Then for $w'(i)$ to have the value $a$, either we have $i=k$ and we passed the target during $(k,k+1)$, or we skipped a different $Path\_C$ step as well.  The latter cannot be, as it would then require another downstep, of which there can only be one for each $T_i$. The prior cannot be true either, as the passing of the target would lead to the need for an aditional downstep as well.
 
 Thus no values between $a$ and $t$ can lie in $w'[i,m_j]$ other than $b$ and $\overline{b}$.  This means that no values in $[\overline{t}]-\{b,\overline{b}\}$ can lie in $w'[i,m_j]$. Since the root $(i,m_j)$ is in $T_{iii}$, we know that $w[1,i-1] = w'[1,i-1]$. Thus $\{1,\hdots,t-1\}\setminus\{b\}\subset\{|w(j)|\}_{1\leq j\leq i}$. We further note that during $T_l,T_{l-1},\hdots,T_1$ none of the values in $[\overline{b}]$ can be transposed out of positions $1$ through $i$ without passing over $t$ in position $i$.  Thus $\{1,\hdots,t\}\subset\{|C'[j]|\}_{1\leq j\leq i}$ as desired.  
 
 We now show that $wt_{i,m_d}[1,i]$ and $C'[1,i]$ have an odd number of $-+$ pairs. Since we have already shown that conditions $2a$ and $2b$ hold and $1\leq b\leq t$, there must be an even number of minus plus pairs between $w't_{i,m_j}$ and $C'$ in the first $i$ positions.  Otherwise, the application of the root $(i,m_j)$ would cause $w't_{i,m_j}[1,k]$ to be blocked off with $C'$ at $i$ by $t$.  This then means that there would have been an odd number of minus plus pairs between $wt_{i,\overline{m_j}}$ and $C'$ in the first $i$ positions.  Thus the use of the root $(i,m_d)$ would have caused the culumns $wt_{i,m_d}[1,k]$ and $C'$ to be blocked off at $i$ by $t$.

\end{proof} 
 
\begin{remark}\label{skip only in stage I} 
{\rm The proof of Lemma~\ref{skip implies block off} also gives us that a skipped $Path\_C$ transposition may only occur during stage I.} 
\end{remark}

\begin{lemma}\label{twice implies block off} Suppose that we have a path such that the root $(k,k+1)$ is used and its application to the word $u$ passes the target in position $k$.  Then the columns $u[1,k]$ and $C'$ are blocked off at $k$ by $C'(k)$.\end{lemma}

\begin{proof}Note that this set up is equivalent to having some $u' = u(k,k+1)$ and skipping the root $(k,k+1)$ which would now be called by the $Path\_C$ algorithm.  Then by Lemma~\ref{skip implies block off}, we know that $u'(k,k+1)[1,k]$ is blocked off with $C'$ at $k$ by $C'(k)$.  The proof concludes with the realization that $u'(k,k+1) = u$.
\end{proof}

\begin{remark}\label{uniqueness of path}
{\rm The Lemmas~\ref{row monotonicity},~\ref{no path for block off},~\ref{skip implies block off}, and \ref{twice implies block off} give conditions that dictate a unique path.}
\end{remark}

\subsection{Constructing a segment of the quantum Bruhat path.}\label{constructing segment of path in B}

In this section we will provide an explicit algorithm for the unique path following the coditions set by Lemmas~\ref{row monotonicity},~\ref{no path for block off},~\ref{skip implies block off}, and`\ref{twice implies block off}.  Further, we will show that this path is the desired path for Proposition~\ref{Path Prop}.

\vspace{12pt} First we distinguish the following two cases:

\begin{enumerate}
\item $C(i)\preceq C'(i) \prec \overline{C'(i)}$

\item $C(i)\preceq\overline{C'(i)}\prec C'(i)$.
\end{enumerate}

We will also need the following notation:

\begin{enumerate}
\item $M_I(u,i,C'):=max(\{u(i)\}\cup\{u(l): k +1\leq l\leq n, u(i)\prec u(l)\preceq C'(i)\})$

\item $M_{III}(u,i,C'):=max(\{\pm u(i)\}\cup\{u(l): k +1\leq l\leq \overline{k+1}, u(i)\prec u(l)\preceq C'(i)\})$
\end{enumerate}

\vspace{12pt}  The following lemma gives a little insight into the nature of $M_I$.  Its proof mirrors a similar lemma given by Lenart in~\cite{Lenart 2012} but cators to type $B_n$.

\vspace{12pt}\begin{lemma}\label{lenart lemma} Under the hypothesis of Proposition~\ref{Path Prop}, in Case 2 we have $$\overline{C'(i)}\preceq M_I(u,i,C')\preceq C'(i).$$\end{lemma}

\vspace{12pt} \begin{proof} Without loss of generality, assume $C'(i)>0$ and that $u(i)\neq C'(i)$.  Let $a = C'(i)\in [n]$ and $A = \{\overline{a},\overline{a-1},\hdots,\overline{1},1,\hdots,a-1,a\}$.  We need to show that $u[k+1,n]$ contains no elements in $A$, so assume the contrary.  Suppose that  $u[i+1,k]=C'[i+1,k]$ contains an element of $A$.  Then by Conditions 1 and 2, $u(i)\in A$ and thus $\overline{C'(i)}\preceq u(i)\preceq M_I(u,i,C')\preceq C'(i)$. Now let $u[i+1,k]=C'[i+1,k]$ not contain any element of $A$.  We conclude that $u[1,i-1]$ contains an element from each pair $\{x,\overline{x}\}$ of elements in $A$.  If $u(i')\in A$ for $i'<i$, we say that $u(i')$ is matched with $C'(i')$.  Since $C'(i) = a$, the only possile matches for $u[1,i-1]\cap A$ are elements in $A\setminus \{a,\overline{a}\}$, by the first part of the reorder condition.  But these are too few to match $a$ elements, which is a contradiction.
\end{proof}

\begin{remark}\label{M_i M_iii order}
{\rm Since $u(i)\preceq M_{I}\preceq M_{III}\preceq C'(i)$, we then have that $\overline{C'(i)}\preceq M_{III}(u,i,C')\preceq C'(i)$.}
 \end{remark}

\vspace{12pt} We now describe the algorithm that constructs the path in Proposition~\ref{Path Prop}.  The algorithm inputs the signed permutation $u$, the target column $C'$, and the position $i$; it outputs the list of reflections $T$ determining the path and the permutation $v$.  It calls on the algorithm \textit{is-blocked-off} which inputs the first $i$ entries of the permutation $u$ and column $C'$ as columns and the value $i$; it then returns whether or not the given columns are blocked off at $i$ by $C'(i)$.  Note that the following procedure is one iteration of Algorithm~\ref{Mod-Greedy algorithm} where we explicitly go through the four stages of $\Gamma^{ki}$. 

\vspace{12pt}
procedure path-B(u,i,$C'$);

\tab Let $c:=C'(i)$;

\tab if $u(i)=c$ then return $\emptyset,u$;

\tab else

\tab\tab Let $S:=\emptyset$, $L:=(k+1,\hdots,n)$ $v:u$;

\tab\tab if $i=k$ and is-blocked-off$(v[1,i],C'[1,i],i)$ then let $S:=S,(k,k+1)$, $v:=v(k,k+1)$, and $L:=L-(k+1)$;

\tab\tab end if;

\tab\tab for m in $L$;

\tab\tab\tab if $v(i)\prec v(m)\preceq c$ and not is-blocked-off$(v(i,m)[1,i],C'[1,i],i)$ then let $S:=S,(i,m)$, $v:=v(i,m)$;

\tab\tab\tab end if;

\tab\tab end for;

\tab\tab let $u_i:=v$; $T_i:=S$

\tab\tab if $sign(u_i)>0$ and $u_i\prec \overline{u_i}\preceq c$ then let $T_{ii}:=T_i,(i,\overline{\imath})$, $u_{ii}:=u_i(i,\overline{\imath})$;

\tab\tab else let $T_{ii}:=T_i$, $u_{ii}:=u_i$;

\tab\tab end if;

\tab\tab let $L:=(\overline{n},\hdots,\overline{k+1})$, $S:=T_{ii}$, $v:=u_{ii}$;

\tab\tab for $m$ in $L$;

\tab\tab\tab if $v(i)\prec v(m)\preceq c$ then let $S:=S,(i,m)$, $v:=v(i,m)$;

\tab\tab\tab end if;

\tab\tab end for;

\tab\tab Let $T_{iii}:=S$, $u_{iii}:=v$, $L:=(\overline{i-1},\hdots,\overline{1})$;

\tab\tab for $m$ in $L$;

\tab\tab\tab if $v(i)\prec v(m)\preceq c$ then let $S:=S,(i,m)$, $v:=v(i,m)$;

\tab\tab\tab end if;

\tab\tab end for;

\tab\tab return($S$,$v$);

\tab end if;

end.

\vspace{.5in}

is-blocked-off(u[1,i],$C'$[1,i],i);

\tab Let $a:= u(i)$, $b:=C'(i)$;

\tab if $a\neq b$ and $|a|\leq b$ and $b>0$

\tab\tab and $\{1,\hdots, b\}\subseteq |u[1,i]|,|
C'[1,i]|$

\tab\tab and $\#\{l: u(l)<0, C'(l)>0\}\% 2 = 1$

\tab\tab then return TRUE;

\tab else return FALSE;

\tab end if;

end. 

\begin{remark}\label{correct termination and uniqueness}
{\rm Condition $1$ gives that $C(l)\neq C'(i)$ for any $l<i$.  This implies that $C'(i)$ appears in some position $q>k$ in $u$. Further, since two columns cannot be blocked off at $i$ with $C(i)=C'(i)$, there is no fear of the algorithm skipping the target value.  This ensures that the algorithm terminates correctly.  The algorithm also respects Lemmas~\ref{row monotonicity},~\ref{no path for block off},~\ref{skip implies block off}, and~\ref{twice implies block off} which together give the uniqueness of the path.}
 \end{remark}

There are two parts of quantum Bruhat criterion in type $B_n$: the sign criterion and the transposing over values criterion.  The following will show that the constructed path follows both.

\begin{lemma}\label{u_i and M_i} If the skip $Path\_C$ step is not called in the algorithm, then $u_i(i)=M_I$.  Otherwise, $u_{iii}(i) = \overline{M_I}$.\end{lemma}

\begin{proof} Suppose that $u_i\neq M_I$, then we have that $M_I\in u_i[k+1,n]$ and $u_i(i)\prec M_I\preceq C'(i)$.  Note that $M_I\neq C'(i)$, as we know that the algorithm terminates correctly.  The remaining strict inequality means that we skipped $M_I$ in the $Path\_C$ algorithm.  By Lemma~\ref{skip implies block off}, this means that the act of transposing the value $M_I$ into position $i$ would have caused the current word to have been blocked off with $C'$ at $i$ by $C'(i)$. Recall from the construction in Lemma~\ref{skip implies block off}, we must have that if $M_I<0$, then $1\preceq \overline{M_I}\preceq u_{iii}(i)\prec t\preceq n$ and if $M_I>0$, then $\overline{n}\preceq \overline{M_I}\preceq u_{iii}(i)\prec M_I \prec t\prec n$.  Further, by the block off condition imposed on the entries in positions $k+1$ through $\overline{k+1}$, we have in both cases that the inequality is really an equality and so $u_{iii} = \overline{M_I}$.
 \end{proof}

\begin{lemma}\label{sign criterion lemma}
Let $u,i,C'$ be inputs for the algorithm and $u_{ii}$ be as given in the algorithm.  Then $\overline{n}\preceq u_{ii}(i)\preceq C'(i)$.
\end{lemma}
 
\begin{proof}
If the skip $Path\_C$ step was not called in the algorithm, then by Lemma~\ref{u_i and M_i} and Lemma~\ref{Lenart reorder lemma} we know that $u_i(i) = M_I \in [\overline{C'(i)},C'(i)]$. If $C'(i)>0$, the result follows.  Otherwise, we may have that $1\preceq C'(i)\prec u_i(i)\prec \overline{n}$, in which case stage II of the algorithm would apply $(i,\overline{i})$ to $u_i$, giving $\overline{n}\preceq u_{ii}(i)\preceq C'(i)$.

If the skip $Path\_C$ step is called at some point in stage I of the algorithm, then we know that $C'(i)>0$ and $C'(i)\prec u(i)\prec \overline{C'(i)}$. Further, the only value from $[\overline{C'(i)}]$ in $u[k+1,n]$ is $M_I$. However, the algorithm skips $M_I$ and so we continue to have $C'(i)\prec u_i(i)\prec \overline{C'(i)}$. After stage II of the algorithm, we can refine this inequality to be $\overline{n}\preceq u_{ii}(i)\prec \overline{C'(i)}\prec C'(i)$.
\end{proof}

\begin{remark}\label{sign criterion remark}
{\rm There is no sign criterion for stage I transpositions and the criterion for stage II is built in to the algorithm. The fact that the procedure uses the $Path\_C$ algorithm in stages III and IV along with Lemma~\ref{sign criterion lemma}  give that stage III and IV upsteps maintain the same sign and downsteps have a sign change of negative to positive. }
\end{remark}

\begin{remarks}\label{pass over criterion remark}
{\rm We now consider the criterion respecting the transposing of the value in position $i$ across the positions $i+1,\hdots,n,\overline{n},\hdots,\overline{k+1}$ and $\overline{i},\hdots,\overline{2}$.
\begin{enumerate}

\item For the positions $i+1,\hdots,k$, recall from the proof of Proposition~\ref{Reorder Nec} that even if there is some $i<l\leq k$ such that $C(i)<C'(l)<C'(i)$, we have that $C'(l)\prec u(i) = CT_kT_{k-1}\hdots T_{i+1}(i)\preceq C'(i)$. Therefore every value in $x\in u[i+1,k]$ is such that $x\prec u(i)\prec C'(i)$. The criterion is then met by Lemma~\ref{row monotonicity} giving the monotonicity of path in position $i$ during the reflections in $T_i$.

\item If we do not skip $Path\_C$ transposition in stage I, then the $Path\_C$ algorithm will be used throughout, and there is no fear of tansposing a value in position $i$ over the positions $k+1,\hdots,n,\overline{n},\hdots,\overline{k+1}$ and $\overline{i},\hdots,\overline{2}$. 

\item If we do skip $Path\_C$ transposition in stage I, then we know from Lemma~\ref{u_i and M_i} that the skipped value was $M_I$. Further, the procedure will only skip a $Path\_C$ step if the corresponding transposition would result in the current word being blocked off with $C'$ at $i$ by $C'(i)$. This means that there are no  values from $[\overline{C'(i)}]$, other than $M_I$ and $\overline{M_I}$, in positions $[k+1,\overline{k+1}]$ in $u,u_i,u_{ii}$ and $u_{iii}$. Thus the only value in positions $[k+1,\overline{k+1}]$ which would lead to possibly transposing over $M_I$ would be $\overline{M_I}$, but this is permitted by the quantum bruhat criterion.  Recall that $u_{iii} = \overline{M_I}$. If $\overline{M_I}=M_{III}$, then the Pat\_C algorithm guarantees that the value in position $i$ will not transpose across positions $[\overline{i},\overline{2}]$ during stage IV transpositions.  If $\overline{M_I}=\overline{M_{III}}$, then we note that transposing $\overline{M_I}$ across $M_I$ is permissible by quantum bruhat criterion, and thereafter the $Path\_C$ algorithm procedure will guarantee that the value in position $i$ will not transpose across positions $[\overline{i},\overline{2}]$ during the remainder of the stage IV transpositions.

 \end{enumerate}
 }
\end{remarks}

\vspace{12pt} The following lemma will further the results from Lemma~\ref{row monotonicity} to all of $T$ rather than just in $T_i$. The proof follows similarly to work in~\cite{Lenart 2012} with some additions concerning type $B_n$.

\begin{lemma}\label{StageIV Monotonicity}
  If such a path hypothesised in Proposition~\ref{Path Prop} exists and $u(l)\neq v(l)$, then $C(l) = u(l)\prec v(l)\preceq C'(l)$ for $l = 1,\hdots,i-1$.
\end{lemma}

\begin{proof}
Suppose that this fails for some $l$ and let $l_1<i$ be the largest such $l$. Let $w$ be the signed permutation to which the reflection $(i,\overline{l_1})$ is applied in stage IV.  Then we have that $$a:=w(i)\prec b:=\overline{C'(l_1)}\prec c_1:=\overline{w(l_1)}.$$ Now let $\widetilde{C}:=w[1,k]$ and $(i,\overline{l_1}),(i,\overline{l_2}),\hdots,(i,\overline{l_p})$ be the remainder of the stage IV roots in $T_i$ where $l_1>l_2>\hdots >l_p$.  If $c_i:=\overline{w(l_i)}$, we then have that $a\prec b\prec c_1\prec c_2\prec \hdots\prec c_p$ by Lemma~\ref{row monotonicity}.

\vspace{12pt} By the quantum Bruhat criterion, there can be no values between $a$ and $c_p$ in $w[i,\overline{l_1}]$.  Therefore, for any $x\in\{b,b+1,\hdots,c_p\},$ either $x$ or $\overline{x}$ is in $\widetilde{C}[1,i]$, say in position $j$.  We claim that the possible values for $C'(j)$ are $\{\pm b, \pm (b+1), \hdots , \pm (c_p-1)\}$.  Note then that overall we have too few choices for these positions, which gives a contradiction. We prove the claim in the following two cases.

\vspace{12pt}\hspace{12pt} Case 1: $\widetilde{C}(j)\in\{b,b+1,\hdots,c_p\}$.  since $w(l_p)=\overline{c_p}$, we know that $x$ is not $c_p$. Also, since $C'(i)=c_p$, we have that $C'(j)\in\{x,x+1,\hdots , c_p-1\}$ by Reorder condition $2$ unless $\widetilde{C}C't_{ij}$ is blocked off at $j$ by $c_p$.  However, since $w$ is in stage IV of $T_i$, we know from Lemma~\ref{u_i and M_i} that $\overline{c_p}\prec a\prec c_p$, but $a = \widetilde{C}(i)$, so the blocked off condition cannot hold.

\vspace{12pt}\hspace{12pt} Case 2:  $\widetilde{C}(j)\in\{\overline{c_p},\hdots,\overline{b}\}$. Then $j\leq l_1$, otherwise one of the roots $(i,\overline{l_r})$ in $T$ will transpose values over $x$ in position $\overline{j}$, breaking the quantum Bruhat criterion.  Since $C'(l_1)=\overline{b}$, we can assume $j<l_1$.  Reorder criterion 2 then gives that $C'(j)\in\{\overline{x},\hdots,\overline{b+1}\}$ unless $\widetilde{C}C't_{jl_1}$ is blocked off at $j$ by $\overline{b}$.  But this cannot be, as $|\widetilde{C}(j)|>\overline{b}$ by assumption.

Thus the claim is proven and the lemma holds.
\end{proof}

\begin{proof}\textit{(Proof of Proposition~\ref{Path Prop})}

Assume that $u(i)\neq C'(i)$.  Then, given Remarks~\ref{correct termination and uniqueness},~\ref{sign criterion remark}, and~\ref{pass over criterion remark}, the only fact left to prove is that the reflections in stage IV satisfy the condition in the corresponding quantum Bruhat graph criterion which refers to the transposition of the value in position $i$ across the positions $\overline{k},\hdots,\overline{i+1}$.  Suppose that at some point in stage IV we apply to the current permutation $w$ a reflection $(i,\overline{l})$, with $l<i$, such that for some $j\in\{i+1,\hdots,k\}$ we have $w(i)\prec w(\overline{\jmath}) = \overline{C'(j)}\prec w(\overline{l})$ or equivalently,
 $$w(l)\prec w(j) = C'(j)\prec w(\overline{\imath}).$$
  Then by Lemma~\ref{StageIV Monotonicity}, we have that
   $$C(l)\preceq w(l)\prec w(\overline{\imath})\preceq C'(l).$$ However Proposition~\ref{Reorder Nec} showed that we have $C(l)\prec C'(j)\prec C'(l)$ only if $CC't_{lj}$ is blocked off at $j$ by $C'(j)$.  But from Lemma~\ref{Lenart reorder lemma}, we know that $C_j(j)\preceq C_j(l)\preceq C'(l)$ where we note that $C_j(j) = w(j) = C'(j)$. Lemma~\ref{StageIV Monotonicity} also gives that $C_j(l)\preceq w(l)\preceq C'(l)$. However we now have that $C_j(j) = w(j)\preceq C_j(l)\preceq w(l)\preceq C'(l)$ and $w(l)\prec w(j)\prec C'(l)$, a contradiction.  Thus no such $j$ exists. This completes the proof that the algorithm constructs a path in the quantum Bruhat graph.
\end{proof}
 
Now that we have shown that there is a unique path to determine each $T_i$ and that our algorithm follows that path, we can now show that the third condition given on two columns $CC'$ from Conditions~\ref{nec conditions for B columns} is indeed necessary to build the entire path given in Proposition~\ref{Total Path Prop}.

\begin{lemma}\label{nec condition 3}
The two columns $CC'$ are not blocked off by $C'[i]$ for any $1\leq i <k$.
\end{lemma}

\begin{proof}
It suffices to show that if $CC'$ is blocked off at some $1\leq i <k$ by $b=C'[i]$, then $CT_kT_{k-1}\hdots T_{i+1}C'$ is blocked off at $i$ as well in which case the result follows from Lemma~\ref{no path for block off}.  Suppose there is such an occurrence.  Then for some transposition in $t_{l\overline{m}}\in T_l$ for $1\leq m\leq i<l\leq k$ and the current word $w$ we have that $wC' $ is blocked off at $i$ but $wt_{l\overline{m}}C'$ is not blocked off at $i$.  Due to Lemma~\ref{Lenart 7.1}, we must only check that $t_{l\overline{m}}$ does not make it so that $wt_{l\overline{\imath}}[i]\in [b,\overline{b+1}]$ and that $w t_{l\overline{m}}[1,i]$ and $C'[1,i]$ do not have an even number of $-+$ pairs.

  For the first, we suppose that $m=i$. Since $wC'$ is blocked off at $i$, it must be that $w[i],w[\overline{\imath}\in [\overline{b},b]$ and $w[l], C'[l]\in [b+1,\overline{b+1}]$.  By monotonicity of path given in Lemma~\ref{row monotonicity} we know that $w[l]\prec w[\overline{\imath}]\prec C'[l]$ but then $\overline{b}\preceq w[i]\prec b = C'[i] \prec \overline{w[l]}=wt_{l\overline{\imath}}$ which contradicts Lemma~\ref{StageIV Monotonicity}.  Thus $m\neq i$.
  
  For the second, we suppose that $1\leq m <i$ and $t_{l\overline{m}}$ is a stage IV downstep.  Then $w[l]<0$ and $wt_{l\overline{m}}>0$.  Lemma~\ref{row monotonicity} gives that $w[l]\prec wt_{l\overline{m}}\prec C'[l]$ and since stage IV transpositions cannot change sign $+$ to $-$, it must be that $C'[l]>0$ as well.  Therefore $w[l]\prec \overline{a} \prec C'[l]$ which means that at some point during $T_i$ we would need to transpose over $\overline{a}$ in position $\overline{i}$, breaking the quantum Bruhat criterion.  Thus there is no such stage IV downstep. 

\end{proof}

The procedure $path$-$B$ does not allow adjacent columns to be blocked off at $i$ while applying transpositions in $T_i$.  The following lemma shows that these same transpositions never cause the columns to be blocked off at any other position either.   

\vspace{25pt} \begin{lemma}\label{no block off above ith row}
Let $u = u_0,u_1,\hdots,u_q = v$ be a path as hypothesised in Proposition~\ref{Path Prop}.  Then $u_j[1,l],C'[1,l]$ is not blocked off at $l$ by $C'(l)$ for any $0\leq j\leq q$ and $1\leq l \leq i$ except possibly for $u_0$ if $i=l=k$.
\end{lemma}

\begin{proof}
This is clearly true for $l=i$ by the procedure $path$-$B$.  Further, this is true for $j=0$ by the Conditions 1 and 2. Suppose that there is some $j>0$ and $l<i$ such that $u_j[1,l],C'[1,l]$ is blocked off at $l$ by $C'(l)$.  Without loss of generality, let $j$ be the minimum and $l$ be the maximum of all such occurences. This means that $u_{j-1}[1,l],C'[1,l]$ is not blocked off at $l$ by $C'(l)$ and $u_j = u_{j-1}t_{i\overline{m}}$ some $1\leq m\leq l$.  We will show that this implies that $u_{j}[1,i],C'[1,i]$ is blocked off at $i$ by $C'(i)$, contradicting procedure $path$-$B$.

  We split this into two cases: one where the transposition $t_{i\overline{m}}$ places a value between $\overline{b}$ and $b$ into position $m$  and a second where the transposition $t_{i\overline{m}}$ changes the number of $-+$ pairs in the first $l$ positions of $u_j,C'$ to be odd.
  
  \vspace{12pt}Case 1: Here we have that $u_{j-1}[i]$ is in $ [\overline{C'[l]},C'[l]]$ but $u_{j-1}[m]$ is not and neither is $C'[i]$ since $u_j$ and $C'$ are blocked off at $l$ by $C'[l]$ by hypothesis.  Also, by Lemma~\ref{row monotonicity}  we have that $u_{j-1}[i]\prec \overline{u_{j-1}[m]}\preceq C'[i]$.
  
  We first show that $C'[i]>0$ and that $|u_{j-1}[i]|\leq C'[i]$.  Note that if $u_{j-1}[i]<0$, then $\overline{u_{j-1}[m]}>0$ since there are no stage IV sign preserving down steps.  Further, this means that $C'[i]>0$ as well, since there are no stage IV sign changing up steps.  Similarly, if $u_{j-1}[i]>0$,  then  $\overline{u_{j-1}[m]}, C'[i] >0$ since all stage IV up steps preserve sign.  So either way, $C'[i]>0$ and it follows that $|u_{j-1}[i]|\leq C'[i]$ since $u_{j-1}[i]$ is in $ [\overline{C'[l]},C'[l]]$ but $C'[i]$ is not.
  
  For the second part of the block off criterion, we not that there are no values between $u_{j-1}[i]$ and $C'[i]$ in $u_{j-1}[i+1,\overline{\imath +1}]$ otherwise the QBG criterion would be broken at some point during the application of transpositions in $T_i$. This, along with the fact that values $\{1,\hdots, C'[l]\}\setminus \{u_{j-1}[i]\}$ are in $u_{j-1}[1,l]$ from the hypothesis, gives that $u_{j-1}[1,i]$ contains the values of $\{1,\hdots, C'[i]\}$ up to absolute value.  Further, by Lemma~\ref{Lenart 7.1} during $T_i$ and the fact that no entry in $\{1,\hdots, C'[i]\}$ can be transposed over $C'[i]$ in position $i$ during $T_{i-1}\hdots T_{1}$ we have that $C'[1,i]$ contains $\{1,\hdots, C'[i]\}$ up to absolute value as well.
  
  To show that $u_{j-1}$ and $C'$ are blocked off at $i$ by $C'[i]$, it only remains to be shown that the number $-+$ pairs in the first $i$ positions of $u_{j-1}$ and $C'$ is odd.  As a starting point, since $u_j$ and $C'$ are blocked off at $l$, it must be that there are an odd number of $-+$ pairs in the first $l$ positions of these two columns.  We claim that there are are no $-+$ pairs in $u_{j-1}[l+1,i-1]$ and $C'[l+1,i-1]$.  Since stage IV moves preserves the signs in both positions $i$ and $m$ or changes both from $-$ to $+$, we get that $u_{j-1}$ and $C'$ also has an odd number of $-+$ pairs in the first $i$ positions.
  
  We conclude by proving the claim.  Suppose that there is some $l<x<i$ such that $u_{j-1}[x]<0$ and $C'[x]>0$. First note that because the transposition $t_{i\overline{m}}$ followed the QBG criterion, it must be that $\overline{n}\preceq u_{j-1}[x]\prec u_{j-1}[m]$.  Also, $C'[i]\prec \overline{u_{j-1}[x]}\preceq n$, otherwise we would transpose over $\overline{u_{j-1}[x]}$ in position $\overline{x}$ at some point during the transpositions in $T_i$ while placing $C'[i]$.  But then $u_{j-1}[x]$ will have to transpose over $\overline{C'[i]}$ in position $\overline{\imath}$ at some point during the application of $T_{i-1}\hdots T_x$ while placing $C'[x]>0$ in position $x$.  This contradics the QBG criterion, and so no such $-+$ pair can exist.

\vspace{12pt} Case 2: Here we show that if the first two block off conditions hold for blocking off $u_{j-1}$ and $C'$ at $l$ then the transposition $t_{i\overline{m}}$ cannot change the number of $-+$ pairs in the first $l$ rows to be odd.  Since there are no changes of positive to negative sign via a stage IV transposition, the change to an odd number of $-+$ pairs must be throught the removal of a $-+$ pair in position $m$.  This means that $u_{j-1}[i]<0$ and  $u_{j}[i]>0$.  By our starting assumption for this case, we know that $C'[l]>0$ and $u_{j-1}[l]$ is in $[\overline{C'[l]},C'[l]]$ but $u_{j-1}[i]<0$ and  $u_{j}[i]>0$ are not. This means that when applying $t_{i\overline{m}}$, $u_{j-1}[i]$ transposes over $\overline{u_{j-1}[l]}$ in position $\overline{l}$, breaking the QBG criterion.
\end{proof}

\begin{proof} \textit{(of Proposition~\ref{main prop})}

The necessity of conditions $1,2,$ and $3$ where shown in Proposition~\ref{Reorder Nec} and Lemma~\ref{nec condition 3}. The uniqueness of path and monotonicity follows from Proposition~\ref{Path Prop}.  To prove that the 3 conditions implies the existence of the chain, we iterate the construction in Proposition~\ref{Path Prop} using Algorithm $Path\_B$ for $i=k,k-1,\hdots,1$.  It remains to be shown that if the 3 necessary conditions hold for $C_iC'$, then they hold for $C_{i-1}C'$ as well. We see that Lemma~\ref{StageIV Monotonicity} gives that the first condition holds as well as that if $C_i[l_1] \prec C'[l_1]\prec C'[l_2]$ then $C_{i-1}[l_1] \prec C'[l_1]\prec C'[l_2]$ for $1\leq l_1< l_2\leq i$.  Now, if there is $1\leq l_1< l_2\leq i$ such that $C_i[l_1] \prec C'[l_2]\prec C'[l_1]$ then we have that $C_i$ and $C't_{l_1l_2}$ are blocked off at $l_1$ by $C'[l_2]$ by hypothesis. Then if $C_{i-1}[l_1] \prec C'[l_2]\prec C'[l_1]$, it must be that  $C_{i-1}$ and $C't_{l_1l_2}$ are blocked off at $l_1$ by $C'[l_2]$ by proof of Lemma~\ref{nec condition 3}. To finish, we acknowledge that $C_{i-1}$ and $C'$ are not blocked off at any $l\leq i-1$ by Lemma~\ref{no block off above ith row}.  
\end{proof}

\section{Proof of Proposition~\ref{SER prop}}\label{L2R}
\subsection{Classifying Split, Extended columns}

We now work towards building a subpath of $\Gamma_r(k)$ between the two height $k$ columns resulting from the split extension of a KN column of height some multiple of $2$ less than $k$.  We do so by first classifying some properties of such pairs of columns, and then showing that these properties are sufficient to build a unique path in the QBG.  For this part, we will follow the work of~\cite{Briggs} where a subcase of such pairs of columns were classified.

\begin{conditions}\label{conditions SER}
{\rm
 Consider the following conditions on a pair of columns $CC'$. 
\begin{enumerate}
\item We have $\{|C(i)|: i = 1,\hdots, k\} = \{|C'(i)|: i = 1,\hdots, k\}$.
\item If $$int(C,C'):=\left(\bigcup\limits_{i=1}^{k}\{j\in [\overline{n}]:C(i)\prec j\prec C'(i)\}\right)\setminus \{\pm C(i): i= 1,\hdots,k\},$$ then we have $int(C,C') = \emptyset$.
\item If $C(i)$ and $C'(i)$ are the same sign, then $C(i)<C'(i)$.  Additionally, there are an even number of entries where $C(i)$ is negative and $C'(i)$ is positive. 
\item $CC'$ follow the Conditions~\ref{nec conditions for B columns}.
\end{enumerate}
}
\end{conditions}

We will first define an \textit{initial matching} on the set of columns $\widehat{KN}_k$ and then show that this  matching satisfies Conditions 1,2, and 3.  We will follow up with a reordering of this matching, and then show that these reordered matchings continue to follow the three conditions as well as Condition 4.  Further, we will later find that these conditions are the necessary requirements to determine a segment of the QBG path.

\begin{definition}\label{initial matching def}{\rm \cite{Briggs}}

We define a pair of columns $BB'$ which we will call the {\rm initial matching}.  Let $B:=\widehat{lA}$ and define $B'$ by matching each value in $B$ from top to bottom with a value in $\widehat{rA}$ as follows:
\begin{enumerate}
\item If $a$ was not involved with the splitting or extending process, match $a$ to itself.
\item If $a\in [n]$ is non-zero and required splitting, match $b$ to $a$ and $\overline{a}$ to $\overline{b}$ where $b\in [n]$ is the value used to split $a$.
\item If $a\in [n]$ is the result of a zero splitting, match $a$ with $\overline{a}$.
\item If $a\in [n]$ is a result of extending, match $\overline{a}$ with $a$. 
\end{enumerate}
\end{definition} 

Since the initial matching is a direct result of the splitting extending algorithms, we will refer to an initial matching and a split extended $KN$ column interchangeably. 

\begin{lemma}\label{initial match has 1,2,3}{\rm \cite{Briggs}}
Any initial matching satisfies Conditions 1,2, and 3.
\end{lemma}

We now give an algorithm which takes an initial matching $BB'$ and produces what we will call a \textit{corrected matching} $CC'$ where $C$ is the increasing colomn of entries from $B$ and $C'$ is a reordering of $B'$.

\begin{definition}
Given a pair of columns $BB'$ defined by an initial matching, we produce a {\rm corrected matching} $CC'$ by the following algorithm:

Let $C:=B$ and $C': = B'$;

for $i$ from $1$ to $k-1$ do

\hspace{24pt} let $j\geq i$ be such that
 $C'(j) = min_{\prec_{C(j)}}\{C'(l):l=i,\hdots,k\}$;

\hspace{24pt} let $C' = C't_{ij}$;

end;
\end{definition}

We also recall the conditions set on two columns in types $A_{n-1}$ and $C_n$.  

\begin{definition}
We will refer to the following two conditions as {\rm Conditions $4'$}. For any pair of indices $1 \leq i < l \leq k$,
\begin{enumerate}

\item  $C(i)\neq C'(l)$

\item and the statement $C(i) \prec C'(l) \prec C'(i)$ is false
\end{enumerate}
\end{definition}

\begin{lemma}\label{corrected matching}
Any corrected matching $CC'$ satisfies Conditions 1,2, 3, and 4.
\end{lemma}

\begin{proof}
It was shown in~\cite{Briggs} that the corrected matching $CC'$ satisfies Conditions 1,2,3, and 4'.  It remains to be shown that $4$ holds.  Note that $4'$ implies that parts 1 and 2 of Condition 4.  We finish by showing that the corrected matching $CC'$ is not blocked off for any $1\leq i \leq k$. Suppose that there is such an $i$ where $CC'$ is blocked off.  Since there must be an odd number of $-+$ pairs in the first $i$ rows, there must be at least one negative value in $C$ in position less than or equal to $i$.  Since $C$ is increasing, it must be that $i=k$.  If not, then $|C[k]|<|C[i]|\leq C'[i]$, which contradicts the block off assumption.  However, $CC'$ cannot be blocked off at $k$ because Condition 3 gives that there are an even number of $-+$ pairs.

\end{proof}

We now have that there is a matching which follows the four conditions for when $C$ is ordered increasingly.  Next, we show that there is a matching that follows the four conditions for any reordering of $C$.

\begin{definition}\label{reorder matching}
Given a corrected matching $CC'$ and $\sigma\in S_k$, we produce a {\rm reordered matching} $DD'$ by the following algorithm:

Let $D:=C\sigma$ and $D': = C'\sigma$;

for $i$ from $1$ to $k-1$ do

\hspace{24pt} let $j\geq i$ be such that
 $D'(j) = min_{\prec_{D(j)}}\{D'(l):l=i,\hdots,k\hspace{4pt}\text{such that}\hspace{4pt}DD't_{il}\hspace{4pt}\text{is not blocked off at}\hspace{4pt} i \}$;

\hspace{24pt} let $D' = D't_{ij}$;

end;

\vspace{12pt}  We will refer to the result of each $i^{th}$ iteration  of the algorithm by $DD'_i$.  We let $CC' = DD'_0$.
\end{definition}

\begin{remark}\label{int at i def}
{\rm For the following lemmas, we consider an equivalent version of condition $2$.  For any two columns $CC'$, let $$int_i(C,C'):= \{j\in [\overline{n}]: C(i)\prec j\prec C'(i)\}\cap\{\pm C(i):i=k+1,\hdots,n\}. $$  Then $int(C,C') = \bigcup\limits_{i=1}^k int_i(C,C')$.}
\end{remark}

The following lemma shows that transpositions, which do not result from block-off avoidence, in the reorder algorithm maintain conditions 1,2 and 3.  the method for proving follows similarly to techniques used in~\cite{Briggs}, as it pertained to a subset of columns where the the block-off condition would not appear.

\begin{lemma}\label{not block off avoiding reorder step}
Let $DD'_{i-1}$ be the result of applying ${i-1}$ iterations of the algorithm in~\ref{reorder matching} to corrected columns $CC'$ and suppose that the $i^{th}$ iteration of of the algorithm, which calls for some transposition $t_{ij}$, does not result from a block off avoidance.  If $int_i(D,D'_{i-1})=int_j(D,D'_{i-1})=\emptyset$ and condition 3 holds for $DD'_{i-1}$, then   $int_i(D,D'_{i})=int_j(D,D'_{i})=\emptyset$ and condition 3 holds for $DD'_{i}$.
\end{lemma}

\begin{proof}
We begin by noting that if $i=j$, then the result follows, as we then have $DD'_{i-1}=DD'_{i}$.  Now assume that $j>i$.  Since the algorithm did not choose $D'_{i-1}[j]$ as an act of block-off avoidence, we have that $D[i]\prec D'_{i-1}[j]\prec D'_{i-1}[i]$.  There are then three places where $D[j]$ could lie within this ordering.  Notice that we cannot have $D[i]\prec D'_{i-1}[j]\prec D[j]\prec D'_{i-1}[i]$, as it would break the hypothesis that $int_i(D,D'_{i-1})=int_j(D,D'_{i-1})=\emptyset$ because $k<n$. It must then be that we have either $ D[j]\prec D[i]\prec D'_{i-1}[j]\prec D'_{i-1}[i]$ or $D[i]\prec D[j]\prec D'_{i-1}[j]\prec  D'_{i-1}[i]$, which we will refer to as case $I$ and case $II$ respectively. Notice that in both cases, we have that $int_{i}(D,D'_{i}), int_j(D,D'_{i})\subseteq int_i(D,D'_{i-1})\bigcup int_j(D,D'_{i-1})=\emptyset $.  It remains to be shown that condition 3 holds for $DD'_{i}$.  This clearly holds for pairs in positions other tha $i$ and $j$, as they are unchanged in this stage of the algorithm and it was assumed that condition 3 held for $DD'_{i-1}$.

\vspace{12pt} We now show the monotonicity in matched pairs of the same sign and first check for position $i$.  Since $k<n$ and   $int_i(D,D'_{i})=\emptyset$, there is some $a\in [n]$ such that $a,\overline{a}\notin [D[i],D'_{i}[i]]$. Suppose that $sgn(D[i]) = sgn(D'_{i}[i])$. If condition 3 does not hold at $i$, we have that $D'_{i}[i]<D[i]$.  If $D[i]\in [n]$, then we must have that $D[i]\prec \overline{a}\prec D'_{i}[i]$, a contradiction to our choice of $a$.  Similarly, if $D[i]\in [\overline{n}]\setminus [n]$, then $D[i]\prec a \prec D'_{i}[i]$, again contradicting  our choice of $a$. Thus $D[i]<D'_{i}[i]$ as desired.  The proof for position $j$ follows similarly.

\vspace{12pt}  We now show that the number of $-+$ pairs remains unchanged after the application of $t_{ij}$ to $D'$.  As before, since $k<n$ and   $int_i(D,D'_{i-1})\bigcup int_j(D,D'_{i-1})=\emptyset$ , there is some $a\in [n]$ such that $a,\overline{a}\notin [D[i],D'_{i-1}[i]]$ and $a,\overline{a}\notin [D[j],D'_{i-1}[j]]$.  This means that the progression through the four values in the orderings of case $I$ and case $II$ above never pass over the values $a$ or $\overline{a}$ and so the sign can only change once (either positive to negative or negative to positive).  If we consider all eight configurations of four $+$ and $-$ values ordered with only one sign change (for example, $+++-$ or $--++$) and compare them to the inequalities of case $I$ and case $II$, we see that in both cases the number of $-+$ pairs is preserved.
\end{proof}

\begin{corollary}\label{non block off avoidance preserves two and three}
Let $DD'_{i-1}$ be the result of applying $i-1$ iterations of the algorithm in~\ref{reorder matching} to corrected columns $CC'$ and suppose that the $i^{th}$ iteration of of the algorithm, which calls for some transposition $t_{ij}$, does not result from a block off avoidance.  If conditions 2 and 3 hold for $DD'_{i-1}$, then conditions 2 and 3 hold for $DD'_{i}$.
\end{corollary}

\begin{lemma}\label{block off avoiding reorder step}
Let $DD'_{i-1}$ be the result of applying $i-1$ iterations of the algorithm in~\ref{reorder matching} to corrected columns $CC'$ and suppose that the $i^{th}$ iteration of of the algorithm, which calls for some transposition $t_{ij}$, results from a block off avoidance.  If conditions 2 and 3 hold for $DD'_{i-1}$, then conditions 2 and 3 hold for $DD'_{\hat{\jmath}}$ for some $i<\hat{\jmath}\leq j$.
\end{lemma}

\begin{proof}
The case where $i=j$ is trivial, so suppose that $j>i$.   Since the algorithm choose $D_{i-1}'[j]$ from a block-off avoidence, we have that $D[i]\prec D'_{i-1}[i]\prec D'_{i-1}[j]$.  There are then three places where $D[j]$ could lie in within this ordering.  If we have $D[j]\prec D[i]\prec D'_{i-1}[i]\prec D'_{i-1}[j]$ or $D[i]\prec D[j]\prec D'_{i-1}[i]\prec D'_{i-1}[j]$, then we again have that $int_i(D,D'_{i}), int_j(D,D'_{i})\subseteq int_i(D,D'_{i-1})\bigcup int_j(D,D'_{i-1})=\emptyset $ and so condition 2 still holds.  Condition 3 then holds via similar argumentation to that in the proof of Lemma~\ref{not block off avoiding reorder step}.

\vspace{12pt} Unfortunately, if we have $D[i]\prec D'_i[i]\prec D[j]\prec D'_i[j]$, we are not guaranteed that $int_j(D,D'_{i})=\emptyset$ or the monotonicity of values in row $j$ of $DD'_{i}$.  For this case, we will first show that the pairity condition of $-+$ pairs still holds for $DD'_{i}$ and that the monotonicity of same sign row entries holds for all rows not $j$. Then we use Lemma~\ref{not block off avoiding reorder step} and the fact that the transposition is avoiding a block off configuration to then show that by the end of the $j^{th}$ iteration of the algorithm, conditions 2 and 3 are restored.

\vspace{12pt}  We first note that $DD'_{i-1}$ must have been blocked off at $i$ by $D'_{i-1}[i]$.  If not, the $i^{th}$ iteration of the algorithm would have called for the $t_{ii}$.  Since $DD'_{i-1}$ was blocked off at $i$ (and therefore an odd number of $-+$ pairs in the first $i$ rows), and condition 3 held for $DD'_{i-1}$ (so there is an even number of $-+$ pairs in the columns), there must be some $-+$ pair in some row $j'>i$.  Since all of the values (up to absolute value) between $\overline{D'_{i-1}[i]}$ and $D'_{i-1}[i]$ must lie in the first $i$ rows of $DD'_{i-1}$, it must be that $\overline{n}\preceq D[j']\prec \overline{D'_{i-1}[i]}\prec D[i] \prec D'_{i-1}[i]\prec D'_{i-1}[j]\preceq D'_{i-1}[j']\preceq n$.  This containment of inequalities along with the fact that condition 3 held for $DD'_{i-1}$ gives that $int_i(D,D'_{i}) \subset int_{j'}(D,D'_{i}) =int_{j'}(D,D'_{i-1}) = \emptyset$.  Furthermore, the algorithm gives that $D'_{i-1}[i]\prec D'_{i-1}[j]\prec D'_{i-1}[j']$, so it must be that $D'_{i-1}[j]>0$. Thus $D'_{i-1}[i]\prec D[j] \prec D'_{i-1}[j]$ are all positive values and so $DD'_{i}$ fails the monotonicity of condition 3 in row $j$.   Since $DD'_{i-1}$ was blocked off at $i$, it must be that $D'_{i-1}[i]>0$ as well.  Thus the number of $-+$ pairs is unchanged by the transposition $t_{ij}$ and so the number of $-+$ pairs in $DD'_{i}$ remains even. Finaly,  since we have from above that $ \overline{D'_{i-1}[i]}\prec D[i] \prec D'_{i-1}[i]\prec D'_{i-1}[j]\preceq n$, we have the monotonicity of same sign values (if it applies) in row $i$ of $DD'_{i}$.  

\vspace{12pt} Let $\hat{\jmath}=min({j,j'})$.  Note that during the $i+1$ through $\hat{\jmath}-1$ iterations of the algorithm, there will never be a block off avoidance move.  Indeed, it would have to be blocked off at some value greater that $D'_{i-1}[i]$, but this value is in position $j$.  We further note that none of these stages of the algorithm will ever call for a transposition with position $j$.  This follows from the fact that none of the pairs in these iterations are $-+$ pairs by our choice of $j'$ and that $D'_{i-1}[i]$ is the minimum positive value below position $i$.  Therefore, by Lemma~\ref{not block off avoiding reorder step}, we have that $DD'_{\hat{\jmath}-1}$ still follows conditions 2 and 3 everywhere but in position $j$.  We now break into two cases, depending on the value of $\hat{\jmath}$.

\vspace{12pt} Suppose that $\hat{\jmath}=j'$.  Then the algorithm will call for the transposition $t_{j'j}$ and if we consider the fact that $int_{j'}(DD'_{j'-1})=\emptyset$ and the ordering of these values given above, we see that both conditions 2 and 3 once again hold.

\vspace{12pt}  Suppose that $\hat{\jmath}=j$.  Then the $j^{th}$ iteration of the algorithm again calls for the transposition $t_{j'j}$ and places some $x$ in a position where $D[j]\prec x \prec D'_{i-1}[j']$ and conditions 2 and 3 are restored.
\end{proof}

We will use the following proposition to prove Proposition~\ref{SER prop}.

\begin{proposition}\label{l2r reorder follows all conditions}
A corrected matching $CC'$ satisfies conditions 1,2,3 and 4 if and only if the corresponding reordered matching $DD'$ satisfies conditions 1,2,3 and 4.
\end{proposition}

\begin{proof}
We first show that if $CC'$ follows the conditions, then so to does the corresponding pair $DD'$.  Condition 1 clearly holds, as the values themselves are unaltered in the columns when converting from a corrected matching to a reordered matching.  Iterations of Corollary~\ref{non block off avoidance preserves two and three} and Lemma~\ref{block off avoiding reorder step} give that conditions 2 and 3 hold for the reordered columns $DD'$.  Condition 4 holds because it is equivalent to the reorder algorithm.  We note here that $DD'$ is not blocked off at $k$ because condition 3 gave that there are an even number of $-+$ pairs in the two columns.  The opposite implication is clear by the same reasoning.
\end{proof}

It remains to be shown that these same four conditions on two columns are exactly the conditions needed to describe split extended $KN$ columns of type $B_n$.

\begin{proposition}\label{Conditions are KN columns}
The set of corrected matchings $DD'$ following conditions 1,2,3 and 4 are in bijection with the split extended reordered $KN$ columns of type $B_n$.
\end{proposition}

\begin{proof}
Note that Proposition~\ref{l2r reorder follows all conditions} reduces the columns to be considered to just those $CC'$ sorted columns.  This case was shown in Theorem 4.4 of~\cite{Briggs}.
\end{proof}

\subsection{Building a segment of the QBG Path between Split, Extended columns}

\begin{lemma}\label{nec and suff of SER conditions}
The pair $DD'$ satisfies Conditions~\ref{conditions SER} if and only if there is a path $u=u_0,u_1,\hdots ,u_p=v$ in the corresponding quantum Bruhat graph such that $v[1,k]=C'$ and the edge labels form a subsequence of $\Gamma_r(k)$. Moreover, the mentioned path is unique, and for each $i= 1,\hdots,k$, we have $$C(i)=u_0(i)\preceq u_1(i)\preceq \hdots\preceq u_p(i) = C'(i).$$
\end{lemma}

\begin{proof}
The necessity of Condition 4 is given by Propositions~\ref{Reorder Nec} and~\ref{nec condition 3}. The necessity of Condition 1 is clear from the selection of roots available in $\Gamma_r(k)$.  The necessity of Conditions 2 and 3 follow from the quantum Bruhat criterion of type $B_n$.  Indeed Condition 2 prevents us from ever transposing over a value in positions $k+1,\hdots,n,\overline{n},\hdots, \overline{k+1}$ while Condition 3 follows from the sign restrictions of stage $IV$ moves.

\vspace{12pt} Now, if the four conditions are satisfied, then we obtain the desired path through the iteration of the $Path\_C$ algorithm (cf. Proposition~\ref{Total Path Prop} and its proof).  Indeed,  Condition 1 assures us that the algorithm terminates correctly even with the limited selection of transpositions in $\Gamma_r(k)$.  The monotonicity of same sign pairs from Condition 3 along with Condition 2 assures us that the algorithm never selections reflections of the form $(i,m)$ or $(i,\overline{m})$ for $m>k$. The pairity condition of $-+$ pairs from Condition 3 assures us that $DD'$ are not blocked off at at $k$ so that there is no need for the root $(k,k+1)$, which is unavailable in $\Gamma_r(k)$.  Condition 4 grants use of Lemma~\ref{no block off above ith row}, so every iteration starts off with two columns which are not blocked off.

\end{proof}

We can now prove Proposition~\ref{SER prop}.

\begin{proof}\textit{(of Proposition~\ref{SER prop} )}
This follows directly from Lemma~\ref{nec and suff of SER conditions} and Proposition~\ref{Conditions are KN columns}.
\end{proof}

\section{The bijection in type $D_n$}\label{D}

We briefly outline the major differences in the type $D_n$ constructions.  First, since KN columns of type $D_n$ have no relation in the ordering of $n$ and $\overline{n}$, the type $D_n$ splitting algorithm ``{\it split\_D}'' begins by converting all $(n,\overline{n})$ pairs in a given column to $0$ values, and then it continues as in type $B_n$ {\rm \cite{lecsbd}}.  There is still need for the extending algorithm, and we use the same one as in type $B_n$ (``{\it extend}''). 
The quantum Bruhat graph criterion in type $D_n$ differs from type $B_n$ in that we no longer have any arrows of the form $(i,\overline{\imath})$, but in return we have less restriction concerning arrows of the form  $(i,\overline{\jmath})$.  This change requires further modifications to the path and reordering algorithms, based on the following \textit{type $D_n$ blocked off} condition.

\begin{definition} We say that columns $C = (l_1,l_2,...,l_k)$ and $C' = (r_1,r_2,...,r_k)$ are {\rm type $D_n$ blocked off at $i$ by $b=r_i$} if and only if 
$C$ and $C'$ are blocked off at $i$ by $b=r_i$, or the following hold:
\begin{enumerate}
\item $ -|l_i| \leq b <0$, where $-|l_i| = b$ if and only if $l_i = \overline{b}$;
\item $\{b,b+1,\ldots,n\}\subset \{|l_1|,|l_2|,...,|l_i|\}$ and $\{b,b+1,\ldots,n\}\subset \{|r_1|,|r_2|,...,|r_i|\}$;
\item  and $|\{j : 1\geq j\geq i, l_j>0, r_j<0\}|$ is odd.
\end{enumerate}

We then define $``{Path\_D}''$ and $``{ord\_D}''$ to be as in type $B_n$, but by replacing ``\textit{blocked off}'' with ``\textit{type $D_n$ blocked off}''.
\end{definition}

\begin{theorem}
The map ``$\mbox{Path\_D}\circ \mbox{ord\_D}\circ \mbox{extend}\circ \mbox{split\_D}$'' is the inverse of the type $D_n$ ``sfill\_D'' map.
\end{theorem}





\begin{thebibliography}{1}

\bibitem{Briggs} C.~Briggs. On {C}ombinatorial {Models} for {K}irillov-{R}eshetikhin {Crystals} of {Type} {B}. \textit{ProQuest Dissertations And Theses}; Thesis (Ph.D.), 79-01(E):56, 2017.
     
\bibitem{Bump and Schilling} D. Bump and A. Schilling. \textit{Crystal Bases: Representations and Combinatorics}. World Scientific, 2017.
     
\bibitem{carter} R. Carter. \textit{Lie Algebras of Finite and Affine Type} Cambridge Studies in Advanced Mathematics, 96. Cambridge University Press, 2005.




\bibitem{FOS} G. Fourier, M. Okado, A. Schilling. Kirillov-Reshetikhin crystals for nonexceptional types. \textit{Adv. Math.}, 222(3):1080-1116, 2009.

\bibitem{Fulton} W. Fulton. \textit{Young Tableaux}. Cambridge University Press, 1996.

\bibitem{Fulton Woodward 2004} W. Fulton and C. Woodward. On the quantum product of Schubert classes. \textit{J. Algebraic Geom.}, 13:641-661, 2004.

\bibitem{Hong and Kang} J.~Hong and S.~Kang. \textit{Introduction to Quantum Groups and Crystal Bases}. Graduate Studies in Mathematics, 42, 2002.



\bibitem{Humphreys} J. Humphreys. \textit{Reflection Groups and Coxeter Groups}. Cambridge University Press, 1990.

\bibitem{Kashiwara 1991} M. Kashiwara. On crystal bases of the q-analogue of universal enveloping algebras. \textit{Duke Math. J.}, 63:465-516, 1991.

\bibitem{kascbm} M. Kashiwara. Crystal bases of modified quantized enveloping algebra. \textit{Duke Math. J.}, 73:383--413, 1994.

\bibitem{Kashiwara Nakashima 1994} M. Kashiwara and T. Nakashima. Crystal graphs for representations of the q-analogue of classical Lie algebras. \textit{J. Algebra}, 165:295-345, 1994.

\bibitem{Kirillov Reshetikhin 1990} N. Kirillov and N. Reshetikhin. Representations of Yangians and multiplicities of the inclusion of the irreducible components of the tensor product of representations of simple Lie algebras. \textit{J. Soviet Math.}, 52:3156-3164, 1990.

\bibitem{Lascoux} A. Lascoux. Cyclic permutations on words, tableaux and harmonic polynomials. \textit{Proceedings of the Hyderabad Conf. on Alg. Groups.}, 323-247, 1991.




\bibitem{Lecouvey 2002} C. Lecouvey. Schensted-type correspondence, plactic monoid, and jeu de taquin for type {$C_n$}. \textit{J. Algebra}, 247: 295-331, 2002.



\bibitem{lecsbd} C.~Lecouvey. Schensted-type correspondence and plactic monoids  for types {$B_n$} and {$D_n$}. \textit{J. Algebraic Combin.}, 18:99--133, 2003.

  

  
\bibitem{Lec Lenart} C. Lecouvey and C. Lenart.  Atomic decomposition of characters and crystals.  \textit{Adv. Math.}, IF 1.494, 2020.
  
\bibitem{Lec Okado Shimozono} C. Lecouvey, M. Okado, and M. Shimozono. Affine crystals, one-dimensional sums and parabolic Lusztig q-analogues. \textit{Math. Zeit.}, 271 no.3–4:819–865, 2012.
  


\bibitem{Lenart 2012} C. Lenart. From Macdonald polynomials to a charge statistic beyond type $A$. \textit{J. Combin. Theory Ser. A}, 119:683-712, 2012.





\bibitem{Lenart Lubovsky 2015b} C. Lenart and A. Lubovsky. A generalization of the alcove model and its applications. \textit{J. Algebraic Combin.}, 41:751-783, 2015.

\bibitem{Lenart Lubovsky 2015a} C. Lenart and A. Lubovsky. A uniform realization of the combinatorial $R$-matrix. \textit{Adv. Math.}, 334:151-183, 2018.



\bibitem{LNSSS} C. Lenart, S. Naito, D. Sagaki, A. Schilling, M. Shimozono. A uniform model for Kirillov-Reshetikhin crystals {I}: Lifting the parabolic quantum Bruhat graph.  \textit{Int. Math. Res. Not.}, 7:1848-1901, 2015.

\bibitem{LNSSS 2016} C. Lenart, S. Naito, D. Sagaki, A. Schilling, M. Shimozono. A uniform model for {K}irillov-{R}eshetikhin crystals {II}: {P}ath
  models and {$P=X$}. \textit{Int. Math. Res. Not.}, 14:4259--4319, 2017.


\bibitem{LNSSS III} C. Lenart, S. Naito, D. Sagaki, A. Schilling, M. Shimozono. A uniform model for Kirillov-Reshetikhin crystals III: Nonsymmetric Macdonald polynomials at $t=0$ and Demazure characters.  \textit{ Transform. Groups}, 22:1041–1079, 2017. 

\bibitem{Lenart Postnikov 2007} C. Lenart and A. Postnikov. Affine Weyl groups in $K$-theory and representation theory.  \textit{Int. Math. Res. Not.}, Art. ID rnm038:1-65, 2007.

\bibitem{Lenart Postnikov 2008} C. Lenart and A. Postnikov. A combinatorial model for crystals of Kac-Moody algebras. \textit{Trans. Amer. Math. Soc.}, 360:4349-4381, 2008.


\bibitem{Lenart Schilling 2011} C. Lenart and A. Schilling. Crystal energy functions via the charge in types $A$ and $C$. \textit{Math. Z.}, 273, 2011.

\bibitem{Lusztig} G. Lusztig. Singlularities, character formulas, and a $q$-analogue of weight multiplicities. \textit{Analyse et topologie sure les espaces singuliers (II-III)}, Ast{\'e}risque 101-102:208-227, 1983.

\bibitem{Nakauyashiki and Yamada} A. Nakayashiki and Y. Yamada. Kostka polynomials and energy functions in solvable lattice models. \textit{Selecta Math. (N.S.)}, 3:547-599, 1997.

\bibitem{Ram Yip 2011} A. Ram and M. Yip. A combinatorial formula for Macdonald polynomials. \textit{Adv. Math.}, 226:309-331, 2011.



\bibitem{Santos} J. M. Santos. Symplectic keys and Demazure atoms in type $C$. \textit{Sém. Lothar. Combin.}, 84B , Art. 49, 2020.

\bibitem{Santos2} J. M. Santos. Symplectic right keys -- Type $C$ Willis' direct way (extended abstract), {\tt arXiv:2104.15000}, to appear in { Proceedings of Formal Power Series and Algebraic Combinatorics}, 2021.

\bibitem{Schilling} A. Schilling. A bijection between type $D^{(1)}_n$ crystals and rigged configurations. \textit{J. Algebra}, 285:292–334, 2005.

\bibitem{Sheats} J. Sheats.  A symplectic jeu de taquin bijection between the tableaux of King and of De Concini. \textit{Trans. Amer. Math. Soc.}, 351:3569–3607, 1999.




\end{thebibliography}
\end{document}